\documentclass[10pt,letterpaper]{article}
\usepackage{amsmath,amssymb,amsthm}
\usepackage{graphicx,color}
\usepackage[hyphens]{url}
\usepackage{dsfont}
\usepackage{booktabs}
\usepackage[normalem]{ulem}
\usepackage{mathtools}
\usepackage[hidelinks,colorlinks=true,linkcolor=blue,citecolor=blue]{hyperref}
\usepackage[nameinlink,capitalize]{cleveref}
\usepackage{multirow}
\usepackage{algorithm}
\usepackage[noend]{algpseudocode}
\usepackage{tikz} 
\usepackage{accents}
\usepackage{bm}
\usepackage{stmaryrd}
\usepackage{nicematrix}
\NiceMatrixOptions{code-for-first-row = \color{lightgray},
 code-for-first-col = \color{lightgray},
}

\usepackage{cancel}

\usepackage{soul} 
\usepackage[square,sort,numbers]{natbib}
\usepackage[sort,nocompress,noadjust]{cite}
\usepackage{enumitem}
\usepackage{subcaption}

\usepackage{subcaption}

\usepackage{tabularx}
\newcolumntype{L}{>{\raggedright\arraybackslash}X}

\numberwithin{equation}{section}
\newtheorem{theorem}{Theorem}[section]

\newtheorem{lemma}[theorem]{Lemma}
\newtheorem{remark}[theorem]{Remark}
\newtheorem{prop}[theorem]{Proposition}

\newtheorem{defin}[theorem]{Definition}

\newcommand{\cA}{\mathcal{A}}

\newcommand{\cT}{\mathcal{T}}

\newcommand{\R}{\mathbb{R}}

\newcommand{\N}{\mathbb{N}}

\DeclareMathOperator*{\Tr}{Tr}

\providecommand{\diag}{\mathrm{diag}}

\DeclareMathOperator*{\argmin}{argmin}

\renewcommand{\leq}{\leqslant}
\renewcommand{\geq}{\geqslant}

\providecommand{\angs}[2]{\left\langle{#1},{#2}\right\rangle}
\providecommand{\norm}[1]{\lVert{#1}\rVert}

\newcommand*{\prox}{\mathrm{prox}}

\newcommand{\blambda}{\bar{\lambda}}
\newcommand{\ulambda}{\underline{\lambda}}

\newcommand{\rec}{\mathrm{rec}}
\newcommand{\sparse}{\mathrm{sp}}
\newcommand{\lr}{\mathrm{lr}}
\newcommand{\walpha}{\widetilde{\alpha}}
\newcommand{\lhat}{\widehat{\lambda}}
\newcommand{\ltilde}{\widetilde{\lambda}}
\newcommand{\mtilde}{\widetilde{\mu}}
\newcommand{\gtilde}{\widetilde{\gamma}}
\newcommand{\stilde}{\widetilde{\sigma}}
\newcommand{\wt}[1]{\widetilde{#1}}

\usepackage[margin=1in]{geometry}

\makeatletter
\def\blfootnote{\gdef\@thefnmark{}\@footnotetext}
\makeatother

\begin{document}
	 \title{Optimized methods for composite optimization: \\ a reduction perspective}
	\author{Jinho Bok \and Jason M. Altschuler}
	\date{Department of Statistics and Data Science, University of Pennsylvania \\[2ex] July 1, 2025}

	\maketitle
	
	\blfootnote{A preliminary version~\citep{proxgdsss} of this paper was accepted at the 38th Conference on Learning Theory (COLT 2025). This paper develops a general framework 
    that recovers the results of \citep{proxgdsss} as a special case.
    Email: \texttt{\string{jinhobok, alts\string}@upenn.edu}.} 
		\vspace{-0.5cm}
\begin{abstract}

Recent advances in convex optimization have leveraged computer-assisted proofs to develop \emph{optimized} first-order methods that improve over classical algorithms. 
However, each optimized method is specially tailored for a particular problem setting, and it is a well-documented challenge to extend optimized methods to other settings due to their highly bespoke design and analysis.
We provide a general framework that derives optimized methods for composite optimization \emph{directly} from those for unconstrained smooth optimization. The derived methods naturally extend the original methods, generalizing how proximal gradient descent extends gradient descent. The key to our result is certain algebraic identities that provide a unified and straightforward way of extending convergence analyses from unconstrained to composite settings. As concrete examples, we apply our framework to establish (1) the phenomenon of stepsize acceleration for proximal gradient descent;
(2) a convergence rate for the proximal optimized gradient method \citep{taylorcomposite} which is faster than FISTA \citep{fista}; (3) a new method that improves the state-of-the-art rate for minimizing gradient norm in the composite setting.

\end{abstract}

\newpage
\tableofcontents

\section{Introduction}
First-order methods such as gradient descent (GD) and Nesterov's accelerated gradient descent (AGD) \citep{agd} are widely popular in modern large-scale optimization. 
Recently, the optimization community has devoted significant effort to developing \emph{optimized} methods which have faster convergence rates than classical algorithms such as GD or AGD for fundamental settings such as unconstrained smooth convex optimization (i.e., $\min_x f(x)$ where $f$ is convex and smooth). A celebrated example is the \emph{optimized gradient method} (OGM) \citep{ogm}: in terms of worst-case convergence rate, this method is not only faster than AGD (by a constant factor), but moreover exactly achieves the best possible rate among all black-box first-order methods \citep{drori17}. Another notable example is \emph{stepsize-accelerated GD}, which incorporates time-varying stepsizes (but no momentum) and attains a provably faster rate than the traditional constant-stepsize GD \citep{altschulermsthesis,gdlongsteps, ap23a, ap23b, gsw24}.

However, it remains largely open whether optimized methods can be extended beyond the settings they were originally designed for.
A central challenge is that the design and analysis of optimized methods are heavily tailored to their particular settings and, as a result, often lack flexible structures or properties that are applicable to other settings. A related challenge is that the design and analysis of optimized methods are based on computer-assisted proofs which are typically difficult to interpret, let alone adapt to other settings. In contrast, simpler algorithms such as GD and AGD have been successfully extended to many settings \citep{karimi, gl16, cdhs17, bh18, km19,  garrigos2023handbook} due to versatile components of their design and analysis, for example the descent lemma~\citep{nesterovbook} and momentum~\citep{agd-ode, accelsurvey, fista, accel-lyapunov}. 
These challenges raise a fundamental question:

\begin{center}
    \emph{Is there a unified way to extend optimized methods to other settings?}
\end{center}

\subsection{Contribution}
We answer this question in the affirmative for the fundamental setting of \emph{composite convex optimization}, i.e., $\min_x f(x) + h(x)$ where $f$ is convex and smooth, and $h$ is convex (but potentially nonsmooth) and can be accessed via its proximal operator $\prox_{\alpha h}(x) = \argmin_z \{\alpha h(z) + \frac{1}{2}\norm{z-x}^2\}$. The composite setting is a strict generalization of unconstrained optimization (by letting $h = 0$) and constrained optimization (by letting $h = \iota_{\mathcal{K}}$ be the indicator of a constraint set $\mathcal{K}$), and captures regularization-based problems prevalent in machine learning, signal processing, and statistics \citep{boyd}.

Conceptually, we develop a reduction-style argument which shows that \emph{optimized methods in the unconstrained setting are also optimized in the composite setting}. This allows us to establish results of the following type:
if an algorithm $\cA$ for unconstrained convex optimization has convergence rate
\begin{align*}
    f(x_n) - f(x_*) &\leq \tau_n \norm{x_0 - x_*}^2\,,
\end{align*}
then there is a \emph{composite extension} $\cT(\cA)$ of this algorithm for composite optimization that has convergence rate
\begin{align*}
    f(x_n) + h(x_n) - f(x_*) - h(x_*) & \leq O(\tau_n) \norm{x_0 - x_*}^2\,.
\end{align*}
This reduction similarly applies to the alternative performance metric of gradient norm.
The new algorithm $\cT(\cA)$ extends the original algorithm $\cA$ in a simple and natural way that generalizes how proximal gradient descent extends gradient descent; see the technical overview section below.

We believe that this unified framework provides a useful viewpoint for studying optimized first-order methods, since it reduces the design and analysis of methods for one setting (composite) to another setting (unconstrained). This is in contrast to prior approaches, where the design and analysis of each optimized method are typically done in a case-by-case fashion for each particular setting; see the discussion of prior work in Section~\ref{subsec:priorworks}.

Our unified framework leads to convergence rates that are competitive with highly optimized algorithms and, in some cases, yields state-of-the-art complexity guarantees and answers open problems. As concrete examples, we apply our framework to establish the following:

\begin{table}
\centering
\begin{tabular}{lcc}
\multicolumn{3}{l}{\textbf{Unconstrained setting}} \\
\toprule
Algorithm $\setminus$ Performance Metric & Objective function & Gradient norm \\
\midrule
Stepsize-accelerated GD & $O(1/n^{\log_2 (1+\sqrt{2})})$ \citep{ap23b} & $O(1/n^{\log_2(1+\sqrt{2})})$ \citep{gsw24} \\
OGM(-G) & $O(1/n^2)$ \citep{ogm} & $O(1/n^2)$ \citep{ogm-g} \\
\cmidrule[\heavyrulewidth]{1-3}

\multicolumn{3}{c}{} \\[-0.3em] %
\multicolumn{3}{c}{\textbf{\textcolor{blue}{\Large $\Downarrow$ \normalsize Composite extension}} (Definition \ref{def:compositeextension})} \\
\multicolumn{3}{c}{} \\[-0.3em] %
\multicolumn{3}{l}{\textbf{Composite setting}} \\

\toprule
Algorithm $\setminus$ Performance Metric & Objective function & Gradient norm \\
\midrule
Stepsize-accelerated proximal GD & \textcolor{blue}{$O(1/n^{\log_2(1+\sqrt{2})})$ (Theorem \ref{thm:sss})} & \textcolor{blue}{$O(1/n^{\log_2(1+\sqrt{2})})$ (Theorem \ref{thm:sss-g})} \\
Proximal OGM(-G) & \textcolor{blue}{$O(1/n^2)$ (Theorem \ref{thm:ogm})} & \textcolor{blue}{$O(1/n^2)$ (Theorem \ref{thm:ogm-g})} \\
\bottomrule
\end{tabular}
\caption{\footnotesize Summary of applications. We develop a unified approach for extending optimized methods from unconstrained to composite settings, extending the asymptotic rate for each combination of algorithm and peformance metric.}\label{table:rates}
\end{table}

\begin{itemize}
    \item \textbf{Stepsize-based acceleration of proximal GD.} We answer the open questions of~\citep{ap23a, gdlongsteps} by showing that stepsize-based acceleration is possible for proximal GD. That is, we show that proximal GD can achieve accelerated rates with a judicious choice of stepsizes---without any other modifications to the algorithm (e.g., momentum). Previous results for stepsize-based acceleration were limited to the setting of GD for unconstrained smooth convex optimization, and it was unknown whether this phenomenon was possible in constrained or composite settings. See~\citep[Section 1]{proxgdsss} for a comprehensive discussion of this open problem and the challenges. We show a rate of $O(1/n^{\log_2(1+\sqrt{2})}) \approx O(1/n^{1.2716})$, which improves over the classical $O(1/n)$ guarantee that is tight for proximal GD with constant stepsizes \citep{tv23}. This asymptotic rate is conjecturally optimal even for the simpler setting of vanilla GD for unconstrained optimization~\citep{ap23a,ap23b,gsw24,gswcomposing}; hence, our rate is also conjecturally optimal. Appealingly, our framework enables us to use the same stepsizes that were developed for accelerating vanilla GD---an approach that is natural, yet was previously unclear how to analyze beyond the original setting.

    \item \textbf{Proximal OGM.} Applying our framework to OGM yields an accelerated rate of $O(1/n^2)$ with nearly-optimal constant factor. This rate is faster than all prior methods except OptISTA, an exactly optimal method with a computer-assisted design that was recently developed specifically for this setting~\citep{optista}.
    
    \item \textbf{Proximal OGM-G.} Applying our framework to OGM-G, a ``gradient norm'' version of OGM~\citep{ogm-g}, yields the state-of-the-art convergence rate for minimizing gradient norm in the composite setting. This rate improves over the previous best guarantee~\citep{hduality} by a factor of nearly $10$.

\end{itemize}
See Table \ref{table:rates} for a summary, and see Sections \ref{sec:app-func} and \ref{sec:app-grad} for formal statements and further details.

\subsection{Overview of framework}\label{subsec:overview}

Here we overview our reduction-style approach for extending first-order methods from unconstrained to composite settings. We focus on the main conceptual ideas here; see Section \ref{sec:mainresult} for formal statements. 

\paragraph*{Design of composite extension.} Let $\cA$ be a first-order method for the unconstrained setting. We propose a simple, unified way of designing an algorithm $\cT(\cA)$ for the composite setting. The derived algorithm $\cT(\cA)$ extends $\cA$ in a way that generalizes how proximal GD extends GD. To explain this, recall that proximal GD updates as $x_{t+1} = \prox_{\alpha_t h}(x_t - \alpha_t \nabla f(x_t))$. By definition of the proximal operator, this is equivalent to
\[x_{t+1} = x_t - \alpha_t(\nabla f(x_t) + s_{t+1}),   \text{ where } s_{t+1} \in \partial h(x_{t+1})\,.
\]
One can view this proximal GD update as GD, except with $\nabla f(x_t)$ replaced by $\nabla f(x_t) + s_{t+1}$. Our proposed extension $\cT(\cA)$ generalizes this: whenever a gradient iterate $\nabla f(x_t)$ appears in the original method $\cA$, we replace it with $\nabla f(x_t) + s_{t+1}$ for the composite setting. 

\paragraph*{Analysis of composite extension.} While the \emph{design} of these composite algorithms is simple and natural, its \emph{analysis} is nontrivial. This is the main content of the paper. 
Our starting point is existing analyses for optimized methods in the unconstrained setting. These have been driven by a powerful technique known as the \emph{performance estimation problem} (PEP) \citep{dt14}; in this framework, a dual solution of a certain semidefinite program provides a proof of a given algorithm's convergence rate. Specifically, in order to establish a convergence rate of the form $f(x_n) - f(x_*) \leq \tau_n \norm{x_0 - x_*}^2$ for an algorithm $\cA$, a dual solution to the semidefinite program provides multipliers $\lambda_{ij} \geq 0$ and a sum-of-squares (SOS) quadratic polynomial $P \geq 0$ such that the following identity holds:
\begin{equation}\label{eq:vanilladual}
	\tau_n\norm{x_0 - x_*}^2 - (f(x_n) - f(x_*)) = \sum_{i, j} \lambda_{ij} Q_{ij} + P\,,
\end{equation}
where $Q_{ij} \geq 0$ is a quadratic polynomial in the iterates $x_i,x_j$ and the first-order information of $f$ at these points. 
This identity certifies the desired convergence rate since the right hand side is nonnnegative.

Our analysis is based on a crucial observation in this template \eqref{eq:vanilladual}: \emph{optimized methods have simple solutions}. In particular, for unconstrained optimized algorithms, the sum-of-squares term is often just a single square---rather than the sum of multiple squares that cannot be collapsed into a single square. In other words, the corresponding quadratic form is of rank $1$. This core property was observed in \citep[Chapter 5]{taylorhabilitation} for many different types of optimized methods, and it is an interesting open question in itself to understand if this phenomenon hints at an underlying general theory for optimized methods; see also the discussion of future work in Section \ref{sec:discussion}. 

Surprisingly, this common structure alone is informative enough for us to characterize a (candidate) dual solution in the composite setting. In this setting, a PEP-based approach tries to find multipliers $\lambda_{ij} \geq 0, \mu_{ij} \geq 0$ and an SOS quadratic polynomial $P \geq 0$ such that the following identity holds:
\begin{equation}\label{eq:proxdual}
	\tau_n \norm{x_0 - x_*}^2 - (f(x_n) + h(x_n)- f(x_*) - h(x_*)) = \sum_{i, j} \lambda_{ij} Q_{ij}^f + \sum_{i, j}\mu_{ij}Q^h_{ij} + P\,,
\end{equation}
where now we have pairs of valid inequalities $Q_{ij}^f \geq 0$ and $Q_{ij}^h \geq 0$ for the first-order information at $f$ and $h$, respectively.
The key difficulty in this extended template \eqref{eq:proxdual} is identifying the solutions, i.e., the quantities $\lambda_{ij}$, $\mu_{ij}$, and $P$. This is in large part because of the additional complexity from $h$, which appears in~\eqref{eq:proxdual} but not in~\eqref{eq:vanilladual}. 
Finding such PEP solutions is a well-documented challenge; for example, the pioneering work on PEP for the composite setting \citep{taylorcomposite} writes:
\begin{center}
	\emph{``... algorithmic analyses using [PEP] are intrinsically limited by our ability to solve semidefinite problems, both numerically ... or analytically .... Therefore, any idea leading to (convex) programs that are easier to solve while maintaining reasonable guarantees would be very advantageous.''}
\end{center}

We overcome this challenge by providing closed-form expressions for the quantities
in the composite identity~\eqref{eq:proxdual} in terms of the 
quantities
in the unconstrained identity~\eqref{eq:vanilladual}. This lets us avoid solving \eqref{eq:proxdual} from scratch, and instead we borrow solutions from \eqref{eq:vanilladual}. In particular, an important starting point is that we use the \emph{same} $\lambda_{ij}$ from the unconstrained setting. While this approach is seemingly straightforward, the corresponding analysis is rather subtle since the corresponding terms---namely, $\lambda_{ij} Q_{ij}$ in \eqref{eq:vanilladual} and $\lambda_{ij} Q_{ij}^{f}$ in \eqref{eq:proxdual}---are \emph{not} equal. Indeed, $Q_{ij}^{f}$ has additional terms involving $h$ in the composite setting (see Definition~\ref{def:coco} for details). Our analysis establishes that these differences can be offset with a careful construction of the multipliers $\mu_{ij}$ and the sum-of-squares term, both of which are new and different from the multipliers $\lambda_{ij}$ and the sum-of-squares term in the unconstrained setting.

\paragraph*{Candidate solution.}
Altogether, this provides a general reduction-style approach of obtaining convergence guarantees, since the resulting formulae for $\lambda_{ij}$, $\mu_{ij}$, $P$ are explicit, satisfy the identity~\eqref{eq:proxdual}, and can be applied to any method with the aforementioned rank-$1$ property. 

\par We refer to this as a ``reduction-style approach'' since it is not an end-to-end reduction. In particular, one still needs to check the feasibility constraints\footnote{Feasibility of $\lambda_{ij} \geq 0$ is automatically guaranteed from the re-use of $\lambda_{ij}$ in our approach.} that $\mu_{ij} \geq 0$ and that $P$ is actually SOS. Verifying these two conditions must be done in an algorithm-specific manner but is conceptually simple because of the closed-form expressions. See Sections~\ref{sec:app-func} and~\ref{sec:app-grad} for multiple examples. We emphasize that this is much simpler than solving~\eqref{eq:proxdual} from scratch---since that requires solving for $\lambda_{ij}$, $\mu_{ij}$, $P$; verifying the identity~\eqref{eq:proxdual}; and checking the feasibility constraints $\lambda_{ij} \geq 0$, $\mu_{ij} \geq 0$, and that $P$ is SOS.

\paragraph*{Sum-of-squares structure.} We comment in particular on the SOS verification of $P$, since our framework provides a conceptually new approach for accomplishing this. Given $\lambda_{ij}$ and $\mu_{ij}$, $P$ is defined as the residual in the identity~\eqref{eq:proxdual}. Verifying that $P$ is SOS is in general a key challenge for PEP-based analyses~\citep{dt14, taylorcomposite, optista}, especially in the composite setting since then $P$ is typically of \emph{high rank}, i.e., any decomposition as a sum of squares (if one exists) requires $\Omega(n)$ squares. Finding such a decomposition is challenging, as that amounts to verifying positive semidefiniteness of the corresponding coefficient matrix, the entries of which have complicated algebraic expressions.

\par We provide a new approach for this verification: we combine certain aspects of the analysis for the unconstrained setting \eqref{eq:vanilladual} and the proximal point method \citep{taylorcomposite}. In the former, as mentioned, the matrix is rank-$1$. In the latter, while the matrix has high rank, it has a simple Laplacian structure which implies positive semidefiniteness. We show that modulo a Schur complement, the matrix in our analysis is a sum of a rank-$1$ matrix and a Laplacian matrix. Ultimately, this reduces checking positive semidefiniteness of $P$ to checking that an explicit rank-$1$ perturbation of a Laplacian matrix remains Laplacian; in our applications, this amounts to simply comparing only a couple entries. We remark that more generally, such a combination of structures from different base algorithms may be useful for PEP-type analyses.

\paragraph*{Performance metrics.} So far, we have focused on convergence rates that are measured in terms of suboptimality of the objective function. Our approach applies in a nearly identical manner when convergence rates are instead measured in terms of approximate stationarity, i.e., making the (sub)gradient norm small. We show results of the following type: if an algorithm $\cA'$ for unconstrained convex optimization has convergence rate
\begin{align*}
	\norm{\nabla f(x_n)}^2 &\leq \tau_n' (f(x_0) - f(x_*))\,,
\end{align*}
then the composite extension $\mathcal{T}(\mathcal{A'})$ has convergence rate
\begin{align*}
	\norm{\nabla f(x_n) + s_n}^2 & \leq O(\tau_n') (f(x_0) + h(x_0) - f(x_*) - h(x_*)) 
\end{align*}
for $s_n \in \partial h(x_n)$. 
Note that we use the same composite extension $\cT$ for these results. 
For this setting of gradient norm minimization, the two building blocks we use for the SOS verification have slightly different forms: for the unconstrained setting, the residual term is zero rather than a single square; and for the proximal point method, the corresponding matrix is diagonally dominant rather than Laplacian~\citep{gu2023tight}. The rest of the analysis is conceptually identical.

\subsection{Prior work}\label{subsec:priorworks}
\paragraph*{Performance estimation problem (PEP).} The  PEP framework, pioneered by~\citep{dt14}, formulates the worst-case performance (i.e., convergence rate) of a given algorithm $\cA$ as a semidefinite program. In this auxiliary optimization problem, the objective function corresponds to the performance metric for $\cA$, and the constraints are on the function class and the iterates being updated by $\cA$. The primal searches over worst-case problem instances, while the dual searches over proofs of the convergence rate as described in the previous section. In many settings of convex optimization, this formulation as a semidefinite program is \emph{tight} in that any proof for convergence rate can be expressed as a solution of the dual \citep{thg17}.

Among numerous applications of PEP (see e.g., \citep{accelsurvey, taylorcomposite}), there have been two main streams of work. One line of work establishes tight rates for existing algorithms, such as gradient descent \citep{dt14, pep-polyak, rotaru2024gdexact, kim2025gdexact}, proximal point/gradient method \citep{taylorcomposite, pep-proximal}, splitting methods \citep{pep-operator}, ADMM \citep{pep-admm}, and Chambolle-Pock \citep{pep-quadratic}, among others. Another line of work develops new optimized methods that are either asymptotically or exactly optimal in each setting, including: smooth (strongly) convex optimization \citep{ogm, ogm-g, item}, nonsmooth convex optimization \citep{kelley, subgdrate, pep-quadub}, composite optimization \citep{optista}, fixed-point iteration \citep{pep-fpi},  minimax optimization \citep{pep-minimax}, and monotone inclusion \citep{pep-monotone}. 
Our framework is situated at the interface of these two directions: it yields nearly tight rates by extending existing optimized methods and their proofs to a new setting. In particular, we provide explicit formulae for the PEP-based proofs (a challenge for the first line of work) and show how optimized methods can be naturally extended to a more general setting (a challenge for the second line of work).

An emerging area of study which uses PEP is stepsize-accelerated GD \citep{altschulermsthesis, daccachemsthesis, eloimsthesis, gdlongsteps, ap23a, ap23b, gsw24, rppa, zhangcomposing, gswcomposing, anytimeaccel}, which seeks to improve the rate of GD by only changing the stepsizes. Classically, this was only known to be possible for minimizing convex quadratics~\citep{young}. The recent line of work extends this improvement beyond the quadratic setting by using time-varying stepsize schedules that are nonmonotone and use exceedingly large steps. In particular, \citep{ap23b} showed the ``silver convergence rate'' $O(1/n^{\log_2 (1+\sqrt{2})})$ which is conjectured to be asymptotically optimal among all possible stepsize schedules. See \citep[Section 1.3]{proxgdsss} for a recent overview of this very active literature on stepsize-based acceleration.

\paragraph*{Composite optimization.} 
Composite optimization arises in many applications due to the flexibility afforded by the non-smooth component $h$---in particular as a regularization penalty. For example, in the common setting where $h$ is the $\ell_1$ norm, proximal GD is known as ISTA \citep{ista}. The seminal work of \citep{fista} introduced FISTA, a widely popular algorithm with $O(1/n^2)$ rate. Since then, different variants of FISTA have been proposed \citep{mfista, fistaanotherlook}; see also \citep[Chapter 10]{beckbook} for an overview. 
Notably, the recent paper \citep{optista} presented an optimized method OptISTA whose convergence rate is faster than FISTA and moreover exactly matches the lower bound for black-box first-order methods. 
Several works have also made progress in the task of gradient norm minimization in the composite setting \citep{fista-g, hduality}, achieving $O(1/n^2)$ rate; however, the optimal constant factor is unknown.

Our framework provides a general way to design and analyze methods in this setting. This leads to new algorithms as well as new rates for existing algorithms. Here, we further contextualize the applications of our framework. Our composite extension of OGM recovers POGM, a method that appeared in the original paper on PEP for composite optimization \citep{taylorcomposite}. However, that paper only presents numerical bounds from PEP; we provide rigorous convergence analyses here. We remark that while the design and analysis are simpler for POGM than OptISTA, it does not achieve the exactly optimal rate. This is not just an artefact of our technique and it was numerically observed that the POGM is suboptimal, though only by a small multiplicative factor of roughly $1.12$ \citep{taylorcomposite}. Our theoretical result achieves a rate for POGM with a mild multiplicative factor of roughly $1.29$ compared to the exactly optimal rate. Furthermore, our composite extension of OGM-G achieves the state-of-the-art rate, improving over the result of \citep[Section D.4]{hduality} by a factor of roughly $9.30$. For stepsize-based acceleration, previous results were known only for GD (in the unconstrained setting), and it was an open problem if this phenomenon extends to projected GD (in the constrained setting) or proximal GD (in the composite setting). This was posed as an open problem in \citep{ap23a, gdlongsteps}; see \citep[Section 1]{proxgdsss} for a detailed discussion of the challenges. We resolve this question by showing that the silver stepsizes for vanilla GD enable the same asymptotic rates for proximal GD.\footnote{Part of these results appeared in the preliminary conference version \citep{proxgdsss} of the present paper. The proofs and rates in \citep{proxgdsss} are mathematically equivalent to Theorem \ref{thm:sss} here, but our proof here is much simpler and based on the general framework developed in this paper. This framework also lets us extend these results both to minimizing gradient norm (Theorem \ref{thm:sss-g}), as well to analyzing momentum-based methods (Theorems~\ref{thm:ogm} and~\ref{thm:ogm-g}).
}

\paragraph*{Connections between algorithms.} A recent line of PEP-related work has investigated relations between different algorithms. For example, \citep{fista-g} identified a common geometric structure in accelerated algorithms, and \citep{flocktogether} identified a common property in accelerated minimax algorithms. \citep{factorsqrt2} established a connection between AGD and OGM, along with certain extensions of OGM. \citep{hduality, mirrorduality} developed the notions of H-duality and mirror duality, which are one-to-one correspondences between certain algorithms, one of which is designed for minimizing the objective function and the other for minimizing the gradient norm.

This paper shares some similarities with H-duality (and mirror duality)~\citep{hduality, mirrorduality}, in that they also reduce the design and analysis of one method to another method. However, the details are quite different. Most importantly, our results are orthogonal---in fact complementary---to theirs as we relate algorithms for \emph{different problem settings with the same performance metric}, whereas H-duality relates algorithms for \emph{the same problem setting with different performance metrics}. There are also other differences. H-duality provides a fully-fledged reduction whereas our framework is not fully black-box; however, our framework is general enough to apply to both stepsize-accelerated GD and OGM, which have markedly different proof structures (in terms of the multipliers $\{\lambda_{ij}\}$), while H-duality only applies to methods with specific proof structures.\footnote{A certain analog of H-duality for stepsize-accelerated GD recently appeared in a different work \citep[Section 3.1]{gswcomposing}. However, it only applies to specific choices of stepsizes.}

\section{Main results}\label{sec:mainresult}

\subsection{Composite extension and preliminaries}
Throughout, we consider first-order methods with iterates $x_0, x_1, \dots, x_n$, where $n$ is the total number of iterations. We let $x_*$ denote an optimal point for the relevant optimization problem. For a smooth convex function $f$ and convex function $h$, we use the shorthands $F:= f+h$, $f_i := f(x_i)$, $g_i := \nabla f(x_i)$, $h_i := h(x_i)$, $F_i := F(x_i)$, and we let $s_i$ denote a subgradient of $h$ at $x_i$, for $i \in \{0, 1, \dots, n, *\}$.\footnote{Unless otherwise specified, $*$ is always included when referring to the set of ``all indices''. $*$ is never included when indices are compared by their values (e.g., $i \geq 0, i \leq n$), and indices $*-1,  *+1$ are defined to be equal to $*$.} Without loss of generality (by normalization), the smoothness parameter of $f$ is assumed to be 1 throughout. Vectors are always vertical, and we denote the sum of the entries of a vector $v$ by $\sum v$.

Our result is stated for a general class of first-order methods. In the literature, this class is known as fixed-step first-order methods.

\begin{defin}[Stepsize matrix and first-order methods]\label{def:vanillamethod}An \emph{$n$-stepsize matrix $H$} is an upper triangular matrix indexed as
\[H := \begin{bmatrix}
    \alpha_{1, 0} & \dots & \alpha_{n, 0} \\
    & \ddots & \vdots \\
    & & \alpha_{n, n-1}
\end{bmatrix}\,,
\]
where $\alpha_{k, k-1} \neq 0$ for all $1 \leq k \leq n$. The \emph{($n$-step) first-order method with $H$} is defined as
\begin{align*}
    x_1 &= x_0 - \alpha_{1, 0}g_0, \\
    x_2 &= x_1 - \alpha_{2, 0}g_0 - \alpha_{2, 1}g_1, \\
    &\ \vdots \\
    x_n &= x_{n-1} - \alpha_{n, 0}g_0 - \dots - \alpha_{n, n-1}g_{n-1}.
\end{align*}
\end{defin}

The condition $\alpha_{k, k-1} \neq 0$ ensures that the iterates are not redundant (i.e., not a linear combination of previous iterates) and $H$ is invertible. 

Given a first-order method for the unconstrained setting (i.e., a stepsize matrix $H$), we define its \emph{composite extension} by replacing $g_t$ with $g_t + s_{t+1}$. 

\begin{defin}[Composite extension]\label{def:compositeextension} 
Consider a first-order method with $n$-stepsize matrix $H$. Its \emph{composite extension} is
    \begin{align*}
    x_1 &= x_0 - \alpha_{1, 0}(g_0+s_1), \\
    x_2 &= x_1 - \alpha_{2, 0}(g_0+s_1) - \alpha_{2, 1}(g_1+s_2), \\
    &\  \vdots \\
    x_n &= x_{n-1} - \alpha_{n, 0}(g_0+s_1) - \dots - \alpha_{n, n-1}(g_{n-1} + s_n).
\end{align*}
\end{defin}

\begin{remark}[Implementation] An equivalent, implementable form of the composite extension (Definition \ref{def:compositeextension}) that only involves gradient and proximal oracles is
\begin{align*}
    x_1 &= \prox_{\alpha_{1, 0}h}(x_0 - \alpha_{1, 0} g_0), \\
    s_1 &= \frac{1}{\alpha_{1, 0}}(x_0 - x_1) - g_0,\\
    x_2 &= \prox_{\alpha_{2, 1}h}(x_1 - \alpha_{2, 0}(g_0+s_1) - \alpha_{2, 1}g_1), \\
    s_2 &= \frac{1}{\alpha_{2, 1}}(x_1 - x_2 - \alpha_{2, 0}(g_0 + s_1)) - g_1, \\
    &\ \vdots \\
    x_n &= \prox_{\alpha_{n, n-1}h}(x_{n-1} - \alpha_{n, 0}(g_0 + s_1) - \dots - \alpha_{n, n-1}g_{n-1}).
\end{align*}
An issue for practical implementation is efficiency: as written, this algorithm requires storing all previous (sub)gradients. Notably, this issue is not specific to our approach and already exists in the unconstrained setting. All of the algorithms in our results can be implemented efficiently; see Sections \ref{sec:app-func} and \ref{sec:app-grad}. 
\end{remark}

Following PEP-based analyses, we make use of \emph{co-coercivities}, which form a complete set of inequalities for certifying convergence rates \citep{thg17}. In the unconstrained setting, we only have a single set of co-coercivities $\{Q_{ij}\}$, whereas in the composite setting we have two sets of co-coercivities $\{Q^f_{ij}\}$ and $\{Q^h_{ij}\}$, respectively corresponding to the first-order information of $f$ and $h$.
\begin{defin}[Co-coercivities]\label{def:coco} Let $H$ be a stepsize matrix and let $\{x_i\}$ be the iterates of the first-order method with $H$. Then define
\[Q_{ij} := f_i - f_j - \angs{g_j}{x_i - x_j} - \frac{1}{2}\norm{g_i - g_j}^2\,.
\]
For the iterates $\{x_i\}$ of the composite extension with $H$, define
\begin{align*}
    Q_{ij}^f &:= f_i - f_j - \angs{g_j}{x_i - x_j} - \frac{1}{2}\norm{g_i - g_j}^2, \\
    Q_{ij}^h &:= h_i - h_j - \angs{s_j}{x_i - x_j}.
\end{align*}
Note here that the iterates $\{x_i\}$ in $Q_{ij}$ and $Q^f_{ij}$ are different, since they are generated by the algorithms in Definition~\ref{def:vanillamethod} (unconstrained) and Definition~\ref{def:compositeextension} (composite), respectively.
\end{defin}

These co-coercivities are sometimes called ``valid inequalities'' since they are positive for any $f$ and $h$. 

\begin{lemma}[{\citep[Theorem 4]{thg17}}] Let $f$ be convex and $1$-smooth, and let $h$ be convex. Then $Q_{ij} \geq 0, Q_{ij}^f \geq 0$ and $Q_{ij}^h \geq 0$ for all $i, j$.
\end{lemma}

\subsection{Algebraic recipes for composite extension}

We state our main technical result, starting with objective function as the performance metric for the convergence rates. As overviewed in Section \ref{subsec:overview}, our result begins with a certificate (i.e., a dual PEP solution) for the unconstrained setting \eqref{eq:vanilla-func}, which has a single square residual. From this, we establish a structured algebraic identity \eqref{eq:prox-func} which provides a candidate solution for the dual PEP in the composite setting. Once it is checked that this candidate is indeed a valid solution via its explicit formulae (see the remark below), we directly obtain a formal proof for the convergence rate.

\begin{theorem}[Composite extension for objective function minimization]\label{thm:main-func} For a first-order method with $H$ and its corresponding cocoercivities $\{Q_{ij}\}$, assume that there exist $\lambda = \{\lambda_{ij}\geq 0: i \in \{0, 1, \dots, n, *\}, j \in \{0, 1, \dots, n\}\}$ and $\gamma = [\gamma_0, \gamma_1, \dots, \gamma_n]$ such that
\begin{equation}\label{eq:vanilla-func}
    \sum_{i, j} \lambda_{ij}Q_{ij} + \frac{1}{2}\norm{x_0 - x_* - \sum_{i=0}^n \gamma_i g_i}^2 = R_n(f_* - f_n) + \frac{1}{2}\norm{x_0 - x_*}^2\,.
\end{equation}
Then for the composite extension with $H$ and its corresponding cocoercivities $\{Q_{ij}^f\}$ and $\{Q_{ij}^h\}$, there exist $\mu = \{\mu_{ij}: i \in \{1, \dots, n, *\}, j \in \{1, \dots, n\}\}, \sigma, S$ (explicitly stated in Definition \ref{def:coef-func}) such that
\begin{equation}\label{eq:prox-func}
\begin{split}
   \sum_{i, j} \lambda_{ij}Q^f_{ij} + \sum_{i, j} \mu_{ij}Q^h_{ij} + \frac{1}{2}\norm{x_0 - x_* - u
   }^2 + \frac{1}{2}\Tr(VSV^T) = R_n(F_* - F_n) + \frac{1}{2}(1 + \xi)\norm{x_0 - x_*}^2\,, 
\end{split}    
\end{equation}
where $V := \begin{bmatrix}
    x_0 - x_* | s_1 | \dots | s_n | s_*
\end{bmatrix}$ is the columnwise-concatenated matrix, $u := \sum_{i=0}^n \gamma_i(g_{i} + s_{i+1})  - \gamma_n s_{n+1} +\sum_{i \in \{1, \dots, n,*\}} \sigma_is_i $, and $S := \begin{bmatrix}
    \xi & v^T \\
    v & L
\end{bmatrix}$ with $\sum v = 0$ and $L$ being Laplacian.\footnote{A symmetric matrix is Laplacian if all nondiagonal entries are nonpositive and all row (column) sums are 0. Laplacian matrices are positive semidefinite, since $x^TLx = \sum_{i<j} (-L_{ij})(x_i - x_j)^2 \geq 0$.} In particular, if
\begin{itemize}
    \item[(i)] $\mu_{ij} \geq 0$ for all $i, j$
    \item[(ii)] $S$ is positive semidefinite (e.g., implied if $\xi = v^TL^{\dagger}v$)
\end{itemize}
then the composite extension with $H$ has the following convergence guarantee:
\[F_n - F_* \leq \frac{1+\xi}{2R_n}\norm{x_0 - x_*}^2\,.
\]
\end{theorem}

The main technical challenge in establishing this theorem is the formulae for the multipliers and sum-of-squares term. For ease of exposition, these formulae are provided later in Definition \ref{def:coef-func}, since this requires additional notation. Given these expressions, the proof is based on matching the coefficients (for linear forms of $\{f_i\}, \{h_i\}$ and quadratic forms of $x_0 - x_*, \{g_i\}, \{s_i\}$) in \eqref{eq:vanilla-func} and \eqref{eq:prox-func}. See Appendix \ref{appendix:proof-main-func} for details.

\begin{remark}[Interpretation and instantiation] 
As overviewed in Section~\ref{subsec:overview}, Theorem \ref{thm:main-func} implies that only two statements---items (i) and (ii)---need to be checked in order to obtain a convergence guarantee for the composite extension. This verification is tractable because $\mu, S$ are given in closed form. 

\par The expression for $\xi$ in Theorem~\ref{thm:main-func}, while providing the tightest possible rate using our framework, requires computing the pseudoinverse of $L$. In our applications, we bypass this by instead choosing $\xi > 0$ as a small constant which only inflates the rate slightly.

For checking (i), we can often use structural properties of $\lambda$ and $H$ from \eqref{eq:vanilla-func}. For checking (ii), the key point is that ``most of $S$'' is already positive semidefinite; note that $S$ is of dimension $(n+2) \times (n+2)$, and $L$ is a positive semidefinite submatrix of dimension $(n+1) \times (n+1)$. In this sense, it suffices to show that the Schur complement $L - \frac{1}{\xi}vv^T$ (i.e., a small perturbation of $L$) is positive semidefinite. A tractable approach for this is to show that it is Laplacian, as we already know that the row and column sums are 0 from $\sum v = 0$. See the applications in Section \ref{sec:app-func} for details and concrete examples of checking (i) and (ii).
    
\end{remark}

We also show an analogous result when the performance metric is the (sub)gradient norm. Here the algebraic structure is simpler, as we begin with the sum-of-squares term being $0$ in the unconstrained setting. 

\begin{theorem}[Composite extension for gradient norm minimization]\label{thm:main-grad} For a first-order method with $H$ and its corresponding cocoercivities $\{Q_{ij}\}$, assume that there exists $\lambda' = \{\lambda_{ij}' \geq 0: i, j \in \{0, 1, \dots, n\}\}$ such that
\begin{equation}\label{eq:vanilla-grad}
    \sum_{i, j} \lambda_{ij}'Q_{ij} = -\frac{R'_n}{2}\norm{g_n}^2 + f_0 - f_n\,.
\end{equation}
Then for the composite extension with $H$ and its corresponding cocoercivities $\{Q_{ij}^f\}$ and $\{Q_{ij}^h\}$, there exist $\mu' = \{\mu_{ij}': i \in \{0, \dots, n\}, j \in \{1, \dots, n\}\}, S'$ (explicitly stated in Definition \ref{def:coef-grad}) such that 
\begin{equation}\label{eq:prox-grad}
   \sum_{i, j} \lambda_{ij}'Q^f_{ij} + \sum_{i, j} \mu_{ij}'Q^h_{ij} + \frac{1}{2}\Tr(V'S'(V')^T) = -\frac{R'_n(1-\xi')}{2}\norm{g_n + s_n}^2 + F_0 - F_n\,,
\end{equation}
where $V' := \begin{bmatrix}
    g_n | s_1 | \dots | s_n
\end{bmatrix}$. Furthermore, if
\begin{itemize}
    \item[(i)] $\mu'_{ij} \geq 0$ for all $i, j$ 
    \item[(ii)] $S'$ is positive semidefinite (e.g., implied if $\xi' = 1 - \frac{\lambda'_{n-1, n} + \lambda'_{n, n-1}}{R'_n}$)
\end{itemize}
then the composite extension with $H$ has the following convergence guarantee:
\[\norm{g_n + s_n}^2 \leq \frac{2}{R_n'(1-\xi')}(F_0 - F_n)\,.
\]
\end{theorem}
The proof is similar to that of Theorem~\ref{thm:main-func} and is provided in Appendix \ref{appendix:proof-main-grad}. Note that whereas the positive semidefinite matrix $S$ in Theorem~\ref{thm:main-func} has Laplacian structure, the matrix $S'$ in Theorem~\ref{thm:main-grad} has diagonally dominant structure.\footnote{A symmetric matrix is diagonally dominant if for each row/column, the absolute value of the diagonal entry is at least the sum of the absolute values of the nondiagonal entries. For example, Laplacian matrices are diagonally dominant. Diagonally dominant matrices are positive semidefinite by the Gershgorin circle theorem.}

\begin{remark}[Performance metric for gradient norm minimization]
    In the literature on gradient norm minimization, results are often stated with respect to $f_0 - f_*$ instead of $f_0 - f_n$ (or in the composite setting, $F_0 - F_*$ rather than $F_0 - F_n$). It is easy to see that bounds with respect to the latter imply corresponding bounds with respect to the former. Furthermore, often the converse is also true, meaning that these performance metrics are essentially equivalent.\footnote{Details: combining \eqref{eq:vanilla-grad} with the inequality $Q_{n*} = f_n - f_* - \frac{1}{2}\norm{g_n}^2 \geq 0$ yields $-\frac{R'_n+1}{2}\norm{g_n}^2 + f_0 - f_* \geq 0$; combining \eqref{eq:prox-grad} with the inequality $Q^f_{n*} + Q^h_{n*} = F_n - F_* - \frac{1}{2}\norm{g_n + s_*}^2 \geq 0$ yields $-\frac{R'_n(1-\xi')}{2}\norm{g_n + s_n}^2 + F_0 - F_* \geq 0$. In many cases, the bounds with respect to $f_0 - f_*$ (or $F_0 - F_*$) are obtained precisely in this way. See also, for example, \citep[equation 3]{hduality}.} We state our results with respect to $f_0 - f_n$ and $F_0 - F_n$ because this yields slightly simpler formulations and also applies to settings where no finite minimizer exists.

    \end{remark}

\begin{remark}[Corollary for related performance metric]
    A standard trick is that by running methods for minimizing objective function and minimizing gradient norm, each for $n/2$ iterations, one obtains a final convergence rate for $n$ steps---comparing gradient norm to \emph{initial distance} (both squared)---which is the product of the two constituent rates, see e.g., \citep{fista-g}. For example, by using the composite extensions of OGM and OGM-G,
    \begin{align*}
        \norm{g_n + s_n}^2 &\leq O\left(\frac{1}{(n/2)^2}\right)(f(x_{n/2}) + h(x_{n/2}) - f(x_*) - h(x_*)) \\
        &\leq O\left(\frac{1}{(n/2)^2}\right) \times O\left(\frac{1}{(n/2)^2}\right) \norm{x_0 - x_*}^2 = O\left(\frac{1}{n^4}\right)\norm{x_0 - x_*}^2\,.
    \end{align*}
    Similarly, for stepsize-accelerated proximal GD, our results yield a rate of $O(1/n^{(\log_2 (1+\sqrt{2}))^2}) \approx O(1/n^{1.6168})$ for the final gradient norm compared to the initial distance.
\end{remark}

\subsection{Notation}

To express convergence rates in a simpler form, we often use the standard notation $a_n \sim b_n$ if $\lim_n (a_n/b_n) = 1$. We write $[A]_{i, j}$ to denote the $(i, j)$ entry of the matrix $A$. 
\par It is also convenient to introduce some notation for triangular matrices corresponding to stepsize parameters. We define $U(a_1, \dots, a_n)$ to be the $n \times n$ upper triangular matrix whose $(i, j)$ entry is $a_i$ if $i = j$ and $1$ if $i < j$. We denote $U(\bm{1}_n)$ as $U_n$; this is the $n \times n$ upper triangular matrix with $1$ in every entry on or above the diagonal. For a $n$-stepsize matrix $H$, we define $\widetilde{H} := HU_n$, i.e., 
\[\widetilde{H} = \begin{bmatrix}
    \widetilde{\alpha}_{1, 0} & \dots & \widetilde{\alpha}_{n, 0} \\
    & \ddots & \vdots \\
    & & \widetilde{\alpha}_{n, n-1}
\end{bmatrix}\,,
\]
where $\walpha_{i, j} := \sum_{k=j+1}^i \alpha_{k, j}$. Then it can be observed that the iterates $x_0, \dots, x_n$ of the first-order method with $H$ satisfy
\begin{align*}
    x_1 &= x_0 - \widetilde{\alpha}_{1, 0}g_0, \\
    x_2 &= x_0 - \widetilde{\alpha}_{2, 0}g_0 - \widetilde{\alpha}_{2, 1}g_1, \\
    &\ \ \vdots \\
    x_n &= x_0 - \widetilde{\alpha}_{n, 0}g_0 - \dots - \widetilde{\alpha}_{n, n-1}g_{n-1}.
\end{align*}

\section{Formulae for multipliers}\label{sec:algebraicform}

In this section, we formally define the multipliers in Theorems \ref{thm:main-func} and \ref{thm:main-grad}. We begin by introducing notation for certain simple linear transformations of the coefficients $\lambda$ and $\mu$ that frequently appear in our analysis. These definitions are stated for $\lambda$ and $\mu$, and apply analogously to $\lambda'$ and $\mu'$.

\begin{defin}[Multipliers] For $\lambda = \{\lambda_{i,j}: i \in \{0, 1, \dots, n, *\}, j \in \{0, 1, \dots, n\}\}$ and $\mu = \{\mu_{i,j}: i \in \{1, \dots, n, *\}, j \in \{1, \dots, n\}\}$, define
\[\lambda_{i, \bullet} := \sum_j \lambda_{i, j}, \quad \lambda_{\bullet, j} := \sum_i \lambda_{i, j}, \quad \mu_{\bullet, j} := \sum_i \mu_{i, j}\,, \\
\]
and
\begin{align*}
	\lhat &:= \begin{bmatrix}
		-(\lambda_{\bullet, 0} + \lambda_{0, \bullet}) & \dots & \lambda_{0, n-2} + \lambda_{n-2, 0} & \lambda_{0, n-1} + \lambda_{n-1, 0} \\
		\vdots & \ddots & \vdots & \vdots \\
		\lambda_{n-2, 0} + \lambda_{0, n-2} & \dots & -(\lambda_{\bullet, n-2} + \lambda_{n-2, \bullet}) & \lambda_{n-2, n-1} + \lambda_{n-1, n-2} \\
		\lambda_{n-1, 0} + \lambda_{0, n-1} & \dots & \lambda_{n-1, n-2} + \lambda_{n-2, n-1} & -(\lambda_{\bullet, n-1} + \lambda_{n-1, \bullet})
	\end{bmatrix} \in \R^{n \times n}, \\
	\ltilde &:= \begin{bmatrix}
		\lambda_{1, 0} & -\lambda_{\bullet, 1} & \lambda_{1, 2} & \dots & \lambda_{1, n-1} \\
		\lambda_{2, 0} & \lambda_{2, 1} & -\lambda_{\bullet, 2} & \dots & \lambda_{2, n-1} \\
		\vdots & \vdots & \ddots & \ddots & \vdots \\ 
		\lambda_{n-1, 0} & \lambda_{n-1, 1} & \dots & \lambda_{n-1, n-2} & -\lambda_{\bullet, n-1} \\
		\lambda_{n, 0} & \lambda_{n, 1} & \dots & \lambda_{n, n-2} & \lambda_{n, n-1}
	\end{bmatrix} \in \R^{n \times n}, \\
	\mtilde &:= \begin{bmatrix}
		-\mu_{\bullet, 1} & \dots & \mu_{1, n} \\
		\vdots & \ddots & \vdots \\
		\mu_{n, 1} & \dots & -\mu_{\bullet, n}
	\end{bmatrix} \in \R^{n \times n}.
\end{align*}
Similar notations are also defined with respect to $\lambda' = \{\lambda_{i, j}': i, j \in \{0, 1, \dots, n\}\}$ and $\mu' = \{\mu_{i, j}': i \in \{0, \dots, n\}, j \in \{1, \dots, n\}\}$.
\end{defin}

\begin{remark}[Reparameterization] Below, it is convenient to define $\mtilde$ and $\mtilde'$ respectively instead of $\mu$ and $\mu'$ (which are the actual multipliers in Theorems \ref{thm:main-func} and \ref{thm:main-grad}). 
	This is equivalent due to the one-to-one correspondence: given $\mtilde$, $\mu_{i, j} = 
	e_i^T \mtilde e_j \cdot \bm{1}_{\{i \neq j\}}$ uniquely defines $\mu$; similarly, given $\mtilde'$, $\mu_{i, j}' = 
	e_i^T \mtilde' e_j \cdot \bm{1}_{\{i \neq j\}}$ uniquely defines $\mu'$, where we use the shorthand $e_0:= -\bm{1}_n$ and $e_*:= -\bm{1}_n$.
\end{remark}

\subsection{Objective function minimization}\label{subsec:form-func}
In the setting of objective function minimization (Theorem \ref{thm:main-func}), the new coefficients $\sigma, \mu, S$ for the composite setting are formally defined as follows.

\begin{defin}\label{def:coef-func} In the setting of Theorem \ref{thm:main-func}, define $\sigma$ as
\begin{equation}\label{eq:sigma}
    \sigma := [\sigma_1, \dots, \sigma_n, \sigma_*], \quad \sigma_i := \frac{\lambda_{i-1, n} + \lambda_{n, i-1}}{\gamma_n}\,.
\end{equation}
For notational convenience, we also define $\gtilde := [\gamma_0, \dots, \gamma_{n-1}], \stilde := [\sigma_1, \dots, \sigma_n]$. Then $\mu, S$ are defined as
    \begin{align}
        \mtilde &:= -\widetilde{H}^{-1}((\widetilde{H}\ltilde)^T + \widetilde{\gamma}(\widetilde{\gamma} + \widetilde{\sigma})^T) = \widetilde{H}^{-1}(\lhat - \gtilde \stilde^T) + \ltilde,
        \label{eq:mutilde-func} \\
        S &:= \begin{bmatrix}
            \xi & v^T \\
            v & L
        \end{bmatrix}, \quad v := [v_1, \dots, v_n, v_*], \quad v_i := \sigma_i + \lambda_{*, i-1} - \mu_{*, i}, \quad L := -\begin{bmatrix}
      \lhat & \gtilde \\
      \gtilde^T & -\lambda_{*, \bullet}
    \end{bmatrix} - \sigma \sigma^T \label{eq:slack-func}.
    \end{align}    
\end{defin}
Note that the equality in \eqref{eq:mutilde-func} is due to \eqref{eq:vanilla-func}; for details, see \eqref{eq:vanilla-func-quad-mat}.

\subsection{Gradient norm minimization}\label{subsec:form-grad}

In the setting of gradient norm minimization (Theorem \ref{thm:main-grad}), the new coefficients $\mu', S'$ for the composite setting are formally defined as follows.

\begin{defin}\label{def:coef-grad} In the setting of Theorem \ref{thm:main-grad}, define $\mu', S'$ as
    \begin{align}
        \mtilde' &:= -\widetilde{H}^{-1}(\widetilde{H}\ltilde')^T, \label{eq:mutilde-grad} \\
        S' &:= \begin{bmatrix}
        R_n' & (v')^T \\
        v' & -\lhat'
    \end{bmatrix} - R'_n(1-\xi')(e_1 + e_{n+1})(e_1 + e_{n+1})^T, \quad v' := [v'_1, \dots, v'_n], \quad v'_i := \lambda_{i-1, n}' + \lambda_{n, i-1}' .\label{eq:slack-grad}
    \end{align}
\end{defin}

\section{Applications to objective function minimization}\label{sec:app-func}
Here we apply our main theorem for objective function minimization (Theorem \ref{thm:main-func}) to the composite extensions of stepsize-accelerated gradient descent and optimized gradient method. For each result, we only need to verify items (i) and (ii) of the theorem. The following subsections do this by using the formulae for the multipliers in Definition \ref{def:coef-func}.

\subsection{Stepsize-accelerated proximal GD}
Gradient descent (GD) is a first-order method with diagonal stepsize matrix $H$. Its composite extension is proximal GD with the same stepsizes.

For $k \in \N$ and $n = 2^k-1$, the silver stepsize schedule $\pi^{(k)}$ (of length $n$) is defined as
\begin{align*}
    \pi^{(k)}_i &:= 1+\rho^{\nu(i)-1} \text{ for } 1 \leq i \leq n\,,
\end{align*}
where $\rho := 1+\sqrt{2}$ is the silver ratio and $\nu(i)$ is the largest integer $j$ such that $2^j$ divides $i$. This can be equivalently defined in a recursive way as $\pi^{(1)} := [\sqrt{2}]$ and $\pi^{(k+1)} := [\pi^{(k)}, \rho^{k-1} + 1, \pi^{(k)}]$. The silver stepsize schedule deviates qualitatively from mainstream stepsize schedules: it is time-varying, nonmonotone, fractal-like, and uses arbitrarily large stepsizes that are in particular larger than $2$ (the threshold at which constant stepsize schedules make GD divergent). See Figure~\ref{fig:sss} for an illustration.

\begin{figure}
	\begin{minipage}[c]{0.4\textwidth}
		\includegraphics[width=\textwidth]{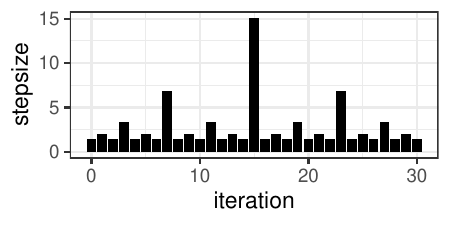}
	\end{minipage}
	\begin{minipage}[l]{0.6\textwidth}
		\caption{\footnotesize First $31$ stepsizes of the silver stepsize schedule. The stepsizes are time-varying and fractal-like.} \label{fig:sss}
	\end{minipage}
\end{figure}

\par ~\citep{ap23b} introduced the silver stepsizes and proved the (partially) accelerated rate of $O(1/n^{\log_2 \rho})$. The following choice of $\lambda$ and $\gamma$ from \citep[Theorem 5.2]{rppa} certifies this asymptotic rate with tight constant factor. The recursive definition of $\lambda$, in essence, combines the proofs for the constitutent stepsize schedules in the recursive definition of $\pi$---a technique known as \emph{recursive gluing} \citep{ap23a, ap23b}.

\begin{prop}[Multipliers for unconstrained GD; {\citep[Theorem 5.2]{rppa}}]\label{prop:sssknown} For $k \in \N$ and $n = 2^k-1$, GD with the silver stepsizes satisfies \eqref{eq:vanilla-func} with
\begin{align*}
    \lambda &= \lambda^{(k)}, \\
    \gamma &= [\pi^{(k)}, \rho^k], \\
    R_n &= 2\rho^k-1.
\end{align*}
Here, $\lambda^{(k)}$ is recursively defined as
\[\lambda_{i, j}^{(k)} := \blambda_{i, j}^{(k)}\bm{1}_{\{i \neq *\}} + \ulambda_j^{(k)}\bm{1}_{\{i = *\}}\,,
\]
with $\ulambda^{(k)} := [\pi^{(k)}, \rho^k]$ and
\begin{align*}
    \blambda^{(1)} &:= \ \begin{bNiceArray}{cc}[first-row, first-col]
        & 0 & 1 \\
        0 & 0 & \rho \\
        1 & 1 & 0 \\
    \end{bNiceArray}, \\
    \blambda^{(k+1)} &:= \underbrace{\blambda^{(k+1), \rec}}_{\text{recursion}} \;\;\;+ \underbrace{\blambda^{(k+1), \sparse}}_{\text{sparse correction}} + \underbrace{\blambda^{(k+1), \lr}}_{\text{low-rank correction}},
\end{align*}
where $\blambda^{(k+1), \rec}, \blambda^{(k+1), \sparse}, \blambda^{(k+1), \lr}$ are defined as
\begin{align*}
    \blambda^{(k+1), \rec}_{i, j} &:= \blambda^{(k)}_{i, j} \bm{1}_{\{0 \leq i, j \leq n\}} + \rho^2 \blambda^{(k)}_{i-n-1, j-n-1}\bm{1}_{\{n+1 \leq i, j \leq 2n + 1\}}, \\
    \blambda^{(k+1), \sparse}_{i, j} &:= \rho \bm{1}_{\{(i, j) = (n, 2n+1)\}} + \rho^k \bm{1}_{\{(i, j) = (2n+1, n)\}}, \\
    \blambda^{(k+1), \lr}_{i, j} &:= \rho \pi^{(k)}_{j-n} \bm{1}_{\{i = n, n+1 \leq j \leq 2n\}} + \rho \pi^{(k)}_{j-n} \bm{1}_{\{i = 2n+1, n+1 \leq j \leq 2n\}}.
\end{align*}
\end{prop}

Based on this, we show that proximal GD with the silver stepsizes also has an $O(1/n^{\log_2 \rho})$ rate. 
This improves over the $O(1/n)$ classical rate of proximal GD \citep[Theorem 7]{tv23} and establishes that the phenomenon of stepsize-based acceleration extends to the composite setting.

\begin{theorem}[Convergence rate of proximal GD]\label{thm:sss}
    For $k \in \N$ and $n = 2^k-1$, proximal GD with the silver stepsizes satisfies
    \[F(x_n) - F(x_*) \leq \frac{\rho}{\sqrt{2}(4\rho^k-2)}\norm{x_0 - x_*}^2 \sim \frac{\rho}{4\sqrt{2}n^{\log_2 \rho}}\norm{x_0 - x_*}^2\,.
    \]
\end{theorem}
\begin{proof} 
	The result follows from an application of Theorem~\ref{thm:main-func} once we verify items (i) and (ii) there. 
	For item (i), from $H = \diag(\gtilde) = \diag(\gamma_0, \dots, \gamma_{n-1})$ we have
\[\mtilde = -\widetilde{H}^{-1}\ltilde^T\widetilde{H}^T - \begin{bmatrix}
    \bm{0}_{(n-1) \times n} \\
    (\gtilde + \stilde)^T
\end{bmatrix}\,.
\] 
The main technical portion is to calculate the first summand. In general, the individual entry of such matrix (in the form $-\widetilde{H}^{-1}A^T\widetilde{H}^T$) can be expressed as a scaled partial sum as follows.

\begin{lemma}\label{lem:partialsum}
    Let $H = \diag(\alpha_1, \dots, \alpha_n)$ where $\alpha_i \neq 0$ for all $1 \leq i \leq n$. Then
    \[[-\widetilde{H}^{-1}A^T\widetilde{H}^T]_{i, j} = \begin{cases}
        \alpha_j \left(\frac{\sum_{l \geq j}[A]_{l, i+1}}{\alpha_{i+1}} - \frac{\sum_{l \geq j}[A]_{l, i}}{\alpha_i}\right) & i \leq n-1 \\
        -\alpha_j\frac{\sum_{l \geq j}[A]_{l, n}}{\alpha_n} & i = n\,.
    \end{cases}
    \]
\end{lemma}
\begin{proof} From $-\widetilde{H}^{-1}A^T \widetilde{H}^T = -U_n^{-1}\diag(1/\alpha_1, \dots, 1/\alpha_n)A^T U_n^T \diag(\alpha_1, \dots, \alpha_n)$ and $[U_n^{-1}]_{i, j} = \begin{cases} 1 & i = j \\ -1 & j = i+1 \\ 0 & \text{else} \end{cases}$,
\[[-U_n^{-1}\diag(1/\alpha_1, \dots, 1/\alpha_n)A^T]_{i, j} = \begin{cases}
     \frac{[A]_{j, i+1}}{\alpha_{i+1}} - \frac{[A]_{j, i}}{\alpha_i} & i \leq n-1 \\
     -\frac{[A]_{j, n}}{\alpha_n} & i = n\,.
\end{cases}
\]
Thus (post)multiplying $U_n^T \diag(\alpha_1, \dots, \alpha_n)$ to this matrix yields the result.
\end{proof}

In this sense, to establish nonnegativity we need to compare partial sums in adjacent columns of $\ltilde$. The key lemma here is that, fortunately, we can compare the individual entries and sum the differences. Below, recall that $[\gamma_0, \dots, \gamma_{n-1}] = \pi^{(k)}$ from Proposition \ref{prop:sssknown}, and $\gamma_n = \rho^k$.

\begin{lemma}\label{lem:ssspartialsum}
    Let $\lambda = \lambda^{(k)}$. Then
    \[\frac{\lambda_{i, j-1}}{\gamma_{j-1}} - \frac{\lambda_{i, j}}{\gamma_j} \begin{cases}
        \geq 0 & 0 \leq i \leq j-2 \\
        \leq 0 & i \geq j+1
    \end{cases}
    \]
    for all $1 \leq j \leq n-1$. In particular, for $j = n$ the inequality is valid for $\gamma_n = 1$.\footnote{For $j = n$, this immediately implies a corresponding inequality for $\gamma_n = \rho^k$. This result is only needed for the induction argument.}
\end{lemma}
\begin{proof}
    We use induction on $k$. For $k = 1$, no such entry exists; for $k = 2$, most of the comparisons are trivial from either $\lambda_{i, j-1} = 0$ or $\lambda_{i, j} = 0$, and the nontrivial ones are: $\frac{\lambda_{3, 1}}{\gamma_1} - \frac{\lambda_{3, 2}}{\gamma_2} = \frac{\rho}{2} - \frac{\rho^2 + \rho + 1}{\sqrt{2}} \leq 0$, $\frac{\lambda_{1, 2}}{\gamma_2} - \frac{\lambda_{1, 3}}{\gamma_3} = \frac{\rho+1}{\sqrt{2}} - \rho = 0$ (letting $\gamma_3 = 1$).

    Now assume that the result holds for $k$, and consider $k+1$ with $\lambda = \lambda^{(k+1)}$ and corresponding $\gamma$. Then for $1 \leq j \leq n$, the inequalities hold from the recursive definition of $\lambda$ and induction hypothesis, with $\frac{\lambda_{2n+1, n-1}}{\gamma_{n-1}} - \frac{\lambda_{2n+1, n}}{\gamma_n} = 0 - \frac{\lambda^{(k+1), \sparse}_{2n+1, n}}{\gamma_n} \leq 0$. For $j = n+1$, note that $\lambda_{i, n} = 0$ for all $n+1 \leq i \leq 2n$ and $\lambda_{i, n+1} = 0$ for all $0 \leq i \leq n-1$; thus it only suffices to consider $i = 2n+1$, where $\frac{\lambda_{2n+1, n}}{\gamma_n} - \frac{\lambda_{2n+1, n+1}}{\gamma_{n+1}} = \frac{\rho^k}{\rho^{k-1} + 1} - \frac{\sqrt{2}\rho}{\sqrt{2}} \leq 0$. For $n+2 \leq j \leq 2n+1$ the inequality readily follows from the induction hypothesis, with $\frac{\lambda_{n, 2n}}{\gamma_{2n}} - \frac{\lambda_{n, 2n+1}}{\gamma_{2n+1}} = \frac{\rho^2}{\rho} - \rho = 0$ (letting $\gamma_{2n+1} = 1$).
\end{proof}

Establishing these, now we return to item (i), i.e., the nonnegativity of $\mu$. First we consider the first $n-1$ rows of $\mtilde$. For $2 \leq i \leq n-1$, we have 
\[\mu_{i, 1} = \gamma_0\left(-\frac{\lambda_{*, i}}{\gamma_{i}} + \frac{\lambda_{*, i-1}}{\gamma_{i-1}}\right) = 0\,,
\]
since $\lambda_{*, j} = \gamma_j$ for all $0 \leq j \leq n-1$. Thus for $1 \leq j < i \leq n-1$, by Lemma \ref{lem:ssspartialsum},
\[\mu_{i, j} = \gamma_{j-1}\left(-\frac{\sum_{l=1}^{j-1}\lambda_{l, i}}{\gamma_i} + \frac{\sum_{l=1}^{j-1} \lambda_{l, i-1}}{\gamma_{i-1}}\right) \geq 0\,,
\]
and similarly for $1 \leq i \leq n-1$ and $j > i$,
\[\mu_{i, j} = \gamma_{j-1}\left(\frac{\sum_{l = j}^n\lambda_{j, i}}{\gamma_i} - \frac{\sum_{l=j}^n \lambda_{j, i-1}}{\gamma_{i-1}} \right) \geq 0\,.
\]
For the $n$th row of $\mtilde$, for $1 \leq j \leq n-1$,
\begin{align*}
    \mu_{n, j} &= \frac{\gamma_{j-1}}{\gamma_{n-1}}(\lambda_{*, n-1} + \sum_{l=1}^{j-1}\lambda_{l, n-1}) - (\gamma_{j-1} + \sigma_j) \\
    &= \frac{\gamma_{j-1}}{\gamma_{n-1}}\sum_{l=1}^{j-1}\lambda_{l, n-1} - \frac{\lambda_{j-1, n} + \lambda_{n, j-1}}{\gamma_n} \\
    &\geq \frac{1}{\gamma_n}(\gamma_{j-1}\sum_{l=1}^{j-1} \lambda_{l, n} - \lambda_{j-1, n} - \lambda_{n, j-1}) \geq 0\,,
\end{align*}
where the second equality is from $\lambda_{*, n-1} = \gamma_{n-1}$ and the definition of $\sigma$, and the final inequality is from the following lemma.

\begin{lemma}
	For $k \geq 2$ and $n = 2^k-1$, $t_j^{(k)} := \pi_{j}^{(k)} \sum_{l=1}^{j-1}\lambda_{l, n}^{(k)} - \lambda_{j-1, n}^{(k)} - \lambda_{n, j-1}^{(k)} \geq 0$
	for all $1 \leq j \leq n-1$.
\end{lemma}
\begin{proof}
    For $k = 2$, $t_1^{(2)} = 0$ and $t_2^{(2)} = 2\lambda_{1, 3}^{(2)} - \lambda_{1, 3}^{(2)} - \lambda_{3, 1}^{(2)} = \rho - \rho = 0$.

    Assume that the result holds for $k$, and consider $2n+1 = 2^{k+1}-1$. Then clearly $t_j^{(k+1)} = 0$ for all $1 \leq j \leq n$, and $t_{n+1}^{(k+1)} = (1+\rho^{k-1})\rho - \rho - \rho^k = 0$. For $n+2 \leq j \leq 2n$,
    \begin{align*}
        t_j^{(k+1)} &=  \pi_j^{(k+1)}\sum_{l=1}^{j-1}\lambda_{l, 2n+1}^{(k+1)} - \lambda_{j-1, 2n+1}^{(k+1)} - \lambda_{2n+1, j-1}^{(k+1)} \\
        &= \pi_{j-n-1}^{(k)}\left(\rho + \sum_{l=1}^{(j-n-1)-1}\rho^2 \lambda_{l, n}^{(k)}\right) - \rho^2\lambda_{(j-n-1)-1, 2n+1}^{(k)} - (\rho^2\lambda_{n, (j-n-1)-1}^{(k)} + \rho\pi_{j-n-1}^{(k)}) \\
        &= \rho^2t^{(k)}_{j-n-1} \geq 0\,.
    \end{align*}
\end{proof}

Furthermore, from $\lambda_{1, 0} = 1$ and $\lambda_{2, 0} = \dots = \lambda_{n, 0} = 0$, $\begin{bmatrix}
        \mu_{*, 1} & \dots & \mu_{*, n}
    \end{bmatrix} = -\bm{1}_n^T \mtilde 
    = e_1 + \gtilde + \stilde$ which proves item (i) with $v_i = \begin{cases}
    -1 & i = 1 \\
    1 & i = * \\
    0 & \text{else}
\end{cases}$.

For item (ii), it suffices to choose $\xi > 0$ such that $L - \frac{1}{\xi}vv^T$ is Laplacian. Since $v = -e_1 + e_{n+1}$, we only need to check that the $(1, n+1)$ entry of this matrix $-\sqrt{2} -\sigma_1\sigma_* + \frac{1}{\xi}$ is nonpositive. For this we can take $\xi = \frac{1}{\sqrt{2}}$, (note that for $k \geq 2$, $\sigma_1 = 0$ and thus the entry becomes 0), proving item (ii).

\end{proof}

\subsection{Proximal OGM}

Optimized gradient method (OGM) \citep{ogm} is a ($n$-step) first-order method with stepsize matrix $H$, whose entry is recursively defined as
\begin{equation}\label{eq:ogmstep}
\alpha_{i+1, j} = \begin{cases}
    \frac{\theta_i-1}{\theta_{i+1}}\alpha_{i, j} & 0 \leq j \leq i-2 \\
    \frac{\theta_i-1}{\theta_{i+1}}(\alpha_{i, i-1}-1) & j = i-1 \\
    1 + \frac{2\theta_i-1}{\theta_{i+1}} & j = i\,.
\end{cases}
\end{equation}
Here, for each $n\in \N$ the sequence $\{\theta_i: 0 \leq i \leq n\}$ is defined as
\begin{equation}\label{eq:theta}
\begin{split}
    \theta_i &= \begin{cases}
        1 & i = 0 \\
    \frac{1+\sqrt{1+4\theta_{i-1}^2}}{2} &  1 \leq i \leq n-1 \\
    \frac{1+\sqrt{1+8\theta_{i-1}^2}}{2} & i = n\,.
    \end{cases}
\end{split}
\end{equation}
It is straightforward to check that $\theta_{i+1}^2 - \theta_{i+1} -\theta_i^2 =0$ for $ 0 \leq i \leq n-2$, $\theta_n^2 - \theta_n - 2\theta_{n-1}^2 = 0$ and $\theta_n^2 \sim n^2/2$ (see e.g., \citep{factorsqrt2}). The certificate for the convergence rate of OGM is rather simple, as presented in the following proposition. 

\begin{prop}[Multipliers for OGM; {\citep[Theorem 2]{ogm}}]
    OGM satisfies \eqref{eq:vanilla-func} with
    \begin{align*}
        \lambda_{i, j} &= \begin{cases}
            2\theta_i^2 & j = i+1, 0 \leq i \leq n-1 \\
            2\theta_j & i = *, 0 \leq j \leq n-1 \\
            \theta_j & i = *, j = n \\
            0 & \text{else}\,,
        \end{cases}\\
        \gamma &= [2\theta_{0}, \dots, 2\theta_{n-1}, \theta_{n}]\,, \\
        R_n &= \theta_n^2 \,.
    \end{align*}
\end{prop}
We establish a comparable rate for its composite extension as follows. Notably, this convergence rate is faster than that of FISTA and only differs by a small constant factor from the exactly optimal rate.

\begin{theorem}[Convergence rate of proximal OGM]\label{thm:ogm}
    For $n \geq 2$, the composite extension of OGM satisfies\footnote{For $n = 1$, $F(x_n) - F(x_*) \leq \frac{2}{3\theta_n^2}\norm{x_0 - x_*}^2 = \frac{1}{6}\norm{x_0 - x_*}^2$ holds, which is tight. For $n \geq 3$, the constant factor in the convergence rate in Theorem~\ref{thm:ogm} can be slightly improved; see the proof and Footnote \ref{footnote:tightconst-func} for details. }
    \[F(x_n) - F(x_*) \leq \frac{3+\sqrt{5}}{8\theta_n^2}\norm{x_0 - x_*}^2 \sim \frac{3+\sqrt{5}}{4n^2}\norm{x_0 - x_*}^2 \,.
    \]
\end{theorem}
\begin{proof}
    As before, we show items (i) and (ii) of Theorem \ref{thm:main-func}. For item (i), first we use the following formulation of the stepsize matrix for OGM.
    \begin{lemma}[Factorization of $H$ for OGM]\label{lem:ogmHmatrix}
        Let $H$ be the $n$-stepsize matrix for OGM. Then 
        \[H = \diag(2\theta_0, \dots, 2\theta_{n-1}) U(\varphi_1, \dots, \varphi_n)  U(\theta_1, \dots, \theta_n)^{-1}\,,
        \]
        where $\varphi_i := \begin{cases}
        1 + \frac{\theta_{i-1}}{2\theta_i} & i < n\ \\
        1 + \frac{\theta_{n-1}}{\theta_n} & i = n\,.
    \end{cases}$
    \end{lemma}
    \begin{proof}
        Consider an equivalent equality (note that each matrix is upper triangular)
        \begin{equation}\label{eq:ogmHmatrix}
            H U(\theta_1, \dots, \theta_n) = \diag(2\theta_0, \dots, 2\theta_{n-1})U(\varphi_1, \dots, \varphi_n)\,.
        \end{equation}
        From \citep[Lemma 4]{ogm}, we have $\sum_{j=k+1}^i \alpha_{j, k} + \theta_{i+1} \alpha_{i+1, k} = 2\theta_k$ for all $0 \leq k < i \leq n-1$, which is precisely the equality corresponding to the $(k+1, i+1)$ entry of \eqref{eq:ogmHmatrix}. Similarly, from the same lemma we have $\theta_{i+1}\alpha_{i+1, i} = \theta_{i+1} + 2\theta_i-1$ for all $0 \leq i \leq n-1$, which is precisely the equality corresponding to the $(i+1, i+1)$ entry of \eqref{eq:ogmHmatrix}, from $2\theta_i \varphi_{i+1} = \theta_{i+1}+2\theta_i-1$ which follows from the recursive definition of $\{\theta_i\}$.
    \end{proof}

    Now we calculate the term $(\widetilde{H}\ltilde)^T$ in \eqref{eq:mutilde-func}. From $$[\ltilde]_{i, j} = \begin{cases}
        -2\theta_i^2 & j = i+1 \\
        2\theta_i^2 & j = i+2 \\
        0 & \text{else}\,,
    \end{cases} \qquad \text{ thus } \qquad [U_n \ltilde]_{i, j} = \begin{cases}
        -2\theta_i^2 & j = i+1 \\
        -2\theta_{j-1} & j \geq i + 2 \\
        0 & \text{else}\,,
    \end{cases}$$ it can be observed that
    \[U_n \ltilde = U(\theta_1, \dots, \theta_n)\begin{bmatrix}
        \bm{0}_{n-1} & \diag(-2\theta_1, \dots, -2\theta_{n-1}) \\
        0 & \bm{0}_{n-1}^T
    \end{bmatrix}\,,
    \]
    which implies
    \begin{align*}
        \widetilde{H}\ltilde &= \diag(2\theta_0, \dots, 2\theta_{n-1})U(\varphi_1, \dots, \varphi_n)\begin{bmatrix}
        \bm{0}_{n-1} & \diag(-2\theta_1, \dots, -2\theta_{n-1}) \\
        0 & \bm{0}_{n-1}^T
    \end{bmatrix} \\
    &= \diag(2\theta_0, \dots, 2\theta_{n-1})\left(\begin{bmatrix}
        \bm{0}_{n-1} & \diag(-\theta_0, \dots, -\theta_{n-2}) \\
        0 & \bm{0}_{n-1}^T
    \end{bmatrix} + \begin{bmatrix}
        & -2\theta_1 & -2\theta_2 & \dots & -2\theta_{n-1} \\
        & & -2\theta_2 & \dots & -2\theta_{n-1} \\
        \bm{0}_{n-1} & & & \ddots & \vdots \\
        & & & & -2\theta_{n-1} \\
        0 & & \bm{0}_{n-1}^T & &
    \end{bmatrix}\right)\,.
    \end{align*}
Also, from \eqref{eq:sigma} we obtain $\sigma = [\bm{0}_{n-1}, \theta_{n}-1, 1]$, which implies that
\[[(\widetilde{H}\ltilde)^T + \gtilde(\gtilde + \stilde)^T]_{i, j} = \begin{cases}
    4\theta_{i-1}\theta_{j-1} & i \leq j \leq n-1 \\
    -2\theta_{j-1}^2 & i = j+1, j \leq n-1 \\
    2(2\theta_{n-1} + \theta_n - 1)\theta_{i-1} & j = n \\
    0 & \text{else} \,.
\end{cases}
\]
Thus 
\begin{align*}
    \mtilde &= -\widetilde{H}^{-1}((\widetilde{H}\ltilde)^T + \gtilde(\gtilde + \stilde)^T) \\
    &= -U_n^{-1}U(\theta_1, \dots, \theta_n)U(\varphi_1, \dots, \varphi_n)^{-1}\diag(2\theta_0, \dots, 2\theta_{n-1})^{-1}((\widetilde{H}\ltilde)^T + \gtilde(\gtilde + \stilde)^T)\,,
\end{align*}
where 
\[[\diag(2\theta_0, \dots, 2\theta_{n-1})^{-1}((\widetilde{H}\ltilde)^T + \gtilde(\gtilde + \stilde)^T)]_{i, j} = \begin{cases}
    2\theta_{j-1} & i \leq j \leq n-1 \\
    -(\theta_j - 1) & i = j+1, j \leq n-1 \\
    2\theta_{n-1} + \theta_n - 1 & j = n \\
    0 & \text{else}\,.
\end{cases}
\]

Now we consider each column of $\mtilde$, which boils down to solving the following equation. The proof is straightforward by solving the equation in the order of $x_n, x_{n-1}, \dots, x_1$.
\begin{lemma}[Solving linear system for $j$th column]\label{lem:ogmls}
    For $1 \leq j \leq n$, define $x(j)$ to be the unique solution $x = [x_1, \dots, x_n]$ of
    \[U(\varphi_1, \dots, \varphi_n)x = 2\theta_{j-1}\begin{bmatrix}
        \bm{1}_{j} \\
        \bm{0}_{n-j}
    \end{bmatrix} -(\theta_{j}-1)e_{j+1}\,,
    \]
    where for here, $e_{n+1} := \bm{0}_n$. Then
    \begin{align*}
        x_n &= \dots = x_{j+2} = 0, \\
        x_{j+1} &= -\frac{\theta_j-1}{\varphi_{j+1}} < 0 , \\
        x_j &= \frac{1}{\varphi_j}(2\theta_{j-1}-x_{j+1}) > 0, \\
        x_i &= \frac{\varphi_{i+1}-1}{\varphi_i}x_{i+1} = \frac{\theta_{i+1}-1}{\theta_{i-1} + 2\theta_i}x_{i+1} > 0, 1 \leq i \leq j-1.
    \end{align*}
\end{lemma}

Now fix $1 \leq j \leq n-1$. Then for $x = x(j)$, by Lemma \ref{lem:ogmls} we have $x_{j+1} < 0$ and $x_i > 0$ for all $1 \leq i \leq j$. Thus with $[U_n^{-1}U(\theta_1, \dots, \theta_n)]_{i, j} = \begin{cases}
    \theta_i & i = j \\
    1 - \theta_j & j = i + 1 \\
    0 & \text{else}\,,
\end{cases}$ we have
\[\begin{bmatrix}
\mu_{1, j} \\
\vdots \\
\mu_{j-1, j} \\
-\mu_{\bullet, j} \\
\mu_{j+1,j }\\
\vdots \\
\mu_{n, j}
\end{bmatrix} = -\begin{bmatrix}
    \theta_1 & 1 - \theta_2 & & &   \\
    & \theta_2 & 1-\theta_3 & & \\
    & & \ddots & \ddots  & \\
    & & & \theta_{n-1} & 1 - \theta_n \\
    & & & & \theta_n
\end{bmatrix}\begin{bmatrix}
    x_1 \\
    \vdots \\
    x_n
\end{bmatrix}\,,
\]
or equivalently,
\begin{align*}
    \mu_{n, j} &= \dots = \mu_{j+2, j} = 0, \\
    \mu_{j+1, j} &= -\theta_{j+1}x_{j+1} > 0, \\
    \mu_{i, j} &= -(\theta_i x_i + (1-\theta_{i+1})x_{i+1}) = -\theta_ix_i + (\theta_{i-1} + 2\theta_i)x_i = (\theta_{i-1} + \theta_i)x_i > 0,  1 \leq i \leq j-1,
\end{align*}
and
\[\mu_{*, j} = \mu_{\bullet, j} - \sum_{i \notin \{j, *\}}\mu_{i, j} = \theta_1x_1 + x_2 + \dots + x_n = (\theta_1 - \varphi_1)x_1 + 2\theta_{j-1} = (\theta_1 - \varphi_1)x_1 + \sigma_j + \lambda_{*, j-1}\,,
\]
which proves that $\mu_{*, j} > 0$ and $v_j < 0$. It can be analogously shown that $\mu_{1, n}, \dots, \mu_{n-1, n} > 0$ and $v_n = \sigma_n + \lambda_{*, n-1} - \mu_{*, n} < 0$, where the only difference is that the factor $2\theta_{j-1}$ in the linear system in Lemma \ref{lem:ogmls} is replaced by $2\theta_{n-1} + \theta_n - 1$, proving item (i).

For item (ii), first note that $L$ is a symmetric matrix such that
\[[L]_{i, j} = \begin{cases}
    4\theta_{i-1}^2 & i = j \leq n-1 \\
    \theta_n^2 - 1 & i = j \in \{n, n+1\} \\
    -2\theta_{i-1}^2 & j = i+1 \leq n \\
    -2\theta_{n-1} - \theta_n + 1 & j=i+1=n+1 \\
    -2\theta_{i-1} & i \leq n-1, j = n+1 \\
    0 & \text{else}\,.
\end{cases}
\]
Also, for all $n \geq 2$, from
\begin{equation}\label{eq:veeone}
\begin{split}
    v_1 &= -2\theta_0(\theta_1 - \varphi_1)\frac{1}{\varphi_1}\left(1+\frac{\theta_1 - 1}{2\theta_0 \varphi_2}\right) \\
    &= -2\left(\frac{1+\sqrt{5}}{2}-\frac{3+\sqrt{5}}{4}\right)(3-\sqrt{5})\left(1+\frac{\sqrt{5}-1}{4\varphi_2}\right) \\
    &= -2(\sqrt{5}-2)\left(1+\frac{\sqrt{5}-1}{4\varphi_2}\right)
\end{split}
\end{equation}
and $1 \leq \varphi_2 < 2$ we have $-\frac{\sqrt{5}-1}{2} \leq v_1 < -\frac{1}{2}$. Thus choosing $\xi = \frac{\sqrt{5}-1}{4}$, $L - \frac{1}{\xi}vv^T$ is Laplacian: it only suffices to check whether the nondiagonal entries on the $(n+1)$th column are still nonpositive, which is true because the $(1, n+1)$ entry is $-2\theta_0 - \frac{v_1}{\xi} \leq -2 + \frac{\sqrt{5}-1}{2\xi} = 0$ and the $(i, n+1)$ entry for $2 \leq i \leq n$ is $-2\theta_{i-1} - \frac{v_i}{\xi} \leq -2 + \frac{1}{2\xi} < 0$ from $-\frac{1}{2} < v_i \leq 0$. Thus $S$ is positive semidefinite, proving item (ii).\footnote{For $n = 1$, from $L = \begin{bmatrix} 3 & -3 \\ -3 & 3
\end{bmatrix}$ and $v_1 = -1$ we can choose $\xi = \frac{1}{3}$. For $n \geq 3$, the smallest value of $\xi$ that our approach can take is $\xi = -\frac{v_1}{2} = \frac{15\sqrt{5}-17 - \sqrt{1942 - 862\sqrt{5}}}{44} \approx 0.2894$ from \eqref{eq:veeone}, slightly smaller than $\frac{\sqrt{5}-1}{4} \approx 0.3090$.\label{footnote:tightconst-func}}
\end{proof}

As noted in Section \ref{subsec:priorworks}, the composite extension of OGM is equal to the proximal OGM (POGM) introduced in \citep{taylorcomposite}, by the following proposition. POGM updates as\footnote{In \citep{taylorcomposite}, an additional notation of $\gamma_k$ is introduced there instead of $\alpha_{k+1, k}$; by \eqref{eq:ogmstep}, they are identical.} (with $y_0 = z_0 = x_0$)
\begin{align*}
    y_{k+1} &= x_{k} - \nabla f(x_k), \\
    z_{k+1} &= y_{k+1} + \frac{\theta_k-1}{\theta_{k+1}}\left(y_{k+1}-y_k + \frac{z_k-x_k}{\alpha_{k, k-1}}\right) + \frac{\theta_{k}}{\theta_{k+1}}(y_{k+1}-x_k), \\
    x_{k+1} &= \prox_{\alpha_{k+1, k}h}(z_{k+1}),
\end{align*}
for $0 \leq k \leq n-1$ (recall that $\alpha_{k, k-1} = 1 + \frac{2\theta_{k-1}-1}{\theta_k}$; for $k = 0$, $\alpha_{k, k-1}$ is not used as $z_k-x_k = 0$), and outputs the final iterate $x_n$.

We note that the equivalence between POGM and the composite extension of OGM is natural, as they share common principles. As explained in \citep[Section 4.3]{taylorcomposite}, POGM is constructed based on two principles: the algorithm being equivalent to OGM when $h \equiv 0$, and the algorithm staying at the optimum point (i.e., $x_{k-1} = x_k = x_*$ implies $x_{k+1} = x_*$). One can deduce that both of these are satisfied by composite extension.

\begin{prop}[Efficient form as POGM]\label{prop:ogmeff} Fix $n \in \N$, and let $\{(\widetilde{x}_k, \widetilde{y}_k, \widetilde{z}_k): 0 \leq k \leq n\}$ be the iterates of POGM and $\{x_k: 0 \leq k \leq n\}$ be the iterates of the composite extension of OGM, with $\widetilde{x}_0 = x_0$. Then $\widetilde{x}_k = x_k$ for all $1 \leq k \leq n$. 
\end{prop}
\begin{proof}
    Define $x_0^{-} := x_0$ and $x_k^{-} := x_k + \alpha_{k, k-1}s_k$ for $1 \leq k \leq n$. We use (strong) induction on $k$ to show that 
    \[\widetilde{x}_k = x_k, \widetilde{z}_k = x_k^{-}
    \]
    for all $k$ (the case $k = 0$ holds by definition).

    For $k = 1$, we have
    \[\widetilde{z}_1 = \widetilde{y}_1 + \frac{\theta_0}{\theta_1}(\widetilde{y}_1 - \widetilde{x}_0) = \widetilde{x}_0 - \left(1 + \frac{\theta_0}{\theta_1}\right) \nabla f(\widetilde{x}_0) = x_0 - \alpha_{1, 0}g_0 = x_1^{-}\,,
    \]
    and thus $\widetilde{x}_1 = \prox_{\alpha_{1, 0}h}(x_1^{-}) = x_1$.
    Assume that the result hold for $1, \dots, k$. Then
    \begin{align*}
        \widetilde{z}_{k+1} &= \widetilde{y}_{k+1} + \frac{\theta_k-1}{\theta_{k+1}}\left(\wt{y}_{k+1}-\wt{y}_k + \frac{\wt{z}_k-\wt{x}_k}{\alpha_{k, k-1}}\right) + \frac{\theta_{k}}{\theta_{k+1}}(\wt{y}_{k+1}-\wt{x}_k) \\
        &= x_k - g_k + \frac{\theta_{k}-1}{\theta_{k+1}}(x_k -x_{k-1} - (g_k - g_{k-1}) + s_k) -\frac{\theta_k}{\theta_{k+1}}g_k \\
        &= x_k + \frac{\theta_{k}-1}{\theta_{k+1}}(g_{k-1} + s_k) - \left(1 + \frac{2\theta_{k} - 1}{\theta_{k+1}}\right)g_k - \frac{\theta_{k}-1}{\theta_{k+1}}\sum_{j=0}^{k-1}\alpha_{k, j}(g_j + s_{j+1}) \\
        &= x_k- \sum_{j=0}^{k-2}\frac{\theta_k-1}{\theta_{k+1}}\alpha_{k, j}(g_j + s_{j+1}) - \frac{\theta_{k}-1}{\theta_{k+1}}(\alpha_{k, k-1}-1)(g_{k-1} + s_k) - \left(1+\frac{2\theta_k-1}{\theta_{k+1}}\right)g_k \\
        &= x_{k+1}^{-}\,,
    \end{align*}
    where the first equality is from the induction hypothesis and the last equality is from the definition of $\alpha_{k+1, j}$. Thus, $\widetilde{x}_{k+1} = x_{k+1}$.
\end{proof}

\section{Applications to gradient norm minimization}\label{sec:app-grad}

We now consider the performance metric of gradient norm rather than objective function. As in  Section \ref{sec:app-func}, we consider composite extensions of both stepsize-accelerated GD and OGM-G (i.e., a gradient norm version of OGM). For each result, it only suffices to prove item (i) of Theorem \ref{thm:main-grad}, as item (ii) is guaranteed by a certain choice of $\xi'$ (see the details therein). As the arguments are fairly similar to that in Section \ref{sec:app-func}, we only present the main conceptual steps and defer the details to Appendix \ref{appendix:appgrad}.

\subsection{Stepsize-accelerated proximal GD}

To the best of our knowledge, the silver stepsize schedule does not admit coefficients $\lambda'$ satisfying \eqref{eq:vanilla-grad}. However, a slightly modified stepsize presented in \citep{gsw24} satisfies the condition. Compared to the silver stepsizes, this stepsize schedule has larger values in its first half.

\begin{defin}
    For $k \in \N$, define $\tau_1 := 4$ and (recall $\rho = 1+\sqrt{2}$)
    \[\tau_{k+1} := \frac{1}{2}\left(\tau_k + 4\rho^k + \sqrt{\tau_k^2 + 8\rho^k \tau_k}\right)\,.
    \]
    Also, define $\eta_k := 1 + \frac{\sqrt{\tau_k^2 + 8\rho^k\tau_k}-\tau_k}{4}$. Finally, define the stepsize $w^{(k)}$ (of length $n = 2^k-1$) as $w^{(1)} = [3/2], $
    \[w^{(k+1)} := [w^{(k)}, \eta_k, \pi^{(k)}]\,.
    \]
\end{defin}

The multipliers $\lambda' = \lambda'^{(k)}$ for stepsize $w^{(k)}$ are similar to those for silver stepsizes in terms of their recursive definition.

\begin{prop}[Multipliers for unconstrained GD; {\citep[Proposition 1]{gsw24}}]\label{prop:gswknown} For $k \in \N$ and $n = 2^k-1$, GD with stepsize $w^{(k)}$ satisfies \eqref{eq:vanilla-grad} with
\begin{align*}
    \lambda' &= \lambda'^{(k)} , \\
    R'_n &= \tau_k - 1.
\end{align*}
Here, $\lambda'^{(k)}$ is recursively defined as
\begin{align*}
    \lambda'^{(1)} &:= \ \begin{bNiceArray}{cc}[first-row, first-col]
        & 0 & 1 \\
        0 & 0 & 2 \\
        1 & 1 & 0 \\
    \end{bNiceArray}, \\
    \lambda'^{(k+1)} &:= \underbrace{\lambda'^{(k+1), \rec}}_{\text{recursion}} \;\;\;+ \underbrace{\lambda'^{(k+1), \sparse}}_{\text{sparse correction}} + \underbrace{\lambda'^{(k+1), \lr}}_{\text{low-rank correction}},
\end{align*}
where $\lambda'^{(k+1), \rec}, \lambda'^{(k+1), \sparse}, \lambda'^{(k+1), \lr}$ are defined as (with $\blambda^{(k)}$ as previously defined in Proposition \ref{prop:sssknown})
\begin{align*}
    \lambda'^{(k+1), \rec}_{i, j} &:= \lambda'^{(k)}_{i, j} \bm{1}_{\{0 \leq i, j \leq n\}} + \frac{\tau_{k+1}}{\rho^{2k}} \blambda^{(k)}_{i-n-1, j-n-1}\bm{1}_{\{n+1 \leq i, j \leq 2n + 1\}}, \\
    \lambda'^{(k+1), \sparse}_{i, j} &:= \frac{\tau_{k+1}}{2\rho^{2k}} \bm{1}_{\{(i, j) = (n, 2n+1)\}} + \left(\frac{\tau_{k+1}}{2\rho^k}-1\right) \bm{1}_{\{(i, j) = (2n+1, n)\}}, \\
    \blambda'^{(k+1), \lr}_{i, j} &:= \frac{\tau_{k+1}}{2\rho^{2k}} \pi^{(k)}_{j-n} \bm{1}_{\{i = n, n+1 \leq j \leq 2n\}} + \frac{\tau_{k+1}}{2\rho^{2k}} \pi^{(k)}_{j-n} \bm{1}_{\{i = 2n+1, n+1 \leq j \leq 2n\}}.
\end{align*}
\end{prop}

Using these multipliers, we obtain an $O(1/n^{\log_2 \rho})$ rate for proximal GD with the same stepsizes. This establishes the phenomenon of stepsize acceleration for the composite setting, when measured in gradient norm---previously only known for the unconstrained setting \citep{gsw24}.

\begin{theorem}[Convergence rate of proximal GD]\label{thm:sss-g}
    For $k \in \N$ and $n = 2^k-1$, proximal GD with stepsize $w^{(k)}$ satisfies
    \[\norm{\nabla f(x_n) + s_n}^2 \leq \frac{2\sqrt{2}}{\tau_k}(F_0 - F_n) \sim \frac{\sqrt{2}(\rho-\sqrt{\rho})}{n^{\log_2\rho}}(F_0 - F_n) \,,
    \]
    where $s_n \in \partial h(x_n)$.
\end{theorem}

\begin{proof}[Proof sketch] The proof is similar to that of Theorem \ref{thm:sss}. A key to that result was a structured inequality in $\lambda^{(k)}$ (Lemma \ref{lem:ssspartialsum}), which also holds for $\lambda'^{(k)}$ from its recursive definition. See Appendix \ref{appendix:sss-g} for details. 
    
\end{proof}

\subsection{Proximal OGM-G}
OGM-G \citep{ogm-g} is a ($n$-step) first-order method with stepsize matrix $H$, whose entry is recursively defined as (recall the definition of $\{\theta_i: 0 \leq i \leq n\}$ from \eqref{eq:theta})
\[\alpha_{i+1, j} = \begin{cases}
    \frac{\theta_{n-j-1}-1}{\theta_{n-j}}\alpha_{i+1, j+1} & 0 \leq j \leq i-2 \\
    \frac{\theta_{n-j-1}-1}{\theta_{n-j}}(\alpha_{i+1, i}-1) & j = i-1 \\
    1 + \frac{2\theta_{n-i-1}-1}{\theta_{n-i}} & j = i\,.
\end{cases}
\]

As in the case of OGM, the multipliers $\lambda'$ for OGM-G are simple. Indeed, these can be considered as certain transformation of the multipliers $\lambda$ for OGM in an appropriate sense \citep{hduality}.

\begin{prop}[Multipliers for OGM-G; {\citep[Theorem 6.1]{ogm-g}}]
    OGM-G satisfies \eqref{eq:vanilla-grad} with
    \begin{align*}
        (1/\theta_n^2)\lambda'_{i, j} &= \begin{cases}
            \frac{1}{2\theta_{n-j}^2} & j = i+1, 0 \leq i \leq n-1 \\
            \frac{1}{2\theta_{n-j-1}^2} - \frac{1}{2\theta_{n-j}^2} & i = n, 1 \leq j \leq n-1  \\
            \frac{1}{2\theta_{n-1}^2} - \frac{1}{\theta_n^2} & i = n, j = 0\\
            0 & \text{else},
        \end{cases}\\
        R_n' &= \theta_n^2-1 .
    \end{align*}
\end{prop}
Using these multipliers, we establish the following convergence guarantee for the composite extension of OGM-G. This bound is almost $1/10$ of the previously known best result in \citep{hduality}.

\begin{theorem}[Convergence rate of proximal OGM-G]\label{thm:ogm-g}
    For $n \geq 2$, the composite extension of OGM-G satisfies\footnote{For $n = 1$, $\norm{\nabla f(x_n) + s_n}^2 \leq \frac{8}{3\theta_n^2}(F_0 - F_n) = \frac{2}{3}(F_0 - F_n) $ holds. See the proof for details.}
    \[\norm{\nabla f(x_n) + s_n}^2\leq \frac{2(\sqrt{5}-1)}{\theta_n^2}(F_0 - F_n) \sim \frac{4(\sqrt{5}-1)}{n^2}(F_0-F_n) \,,
    \]
    where $s_n \in \partial h(x_n)$.
\end{theorem}
\begin{proof}[Proof sketch] The overall proof structure is similar to that of proximal OGM (Theorem \ref{thm:ogm}). The proof is more technical here (although straightforward) due to each component in the formula $\mu' = \ltilde' + \widetilde{H}^{-1}\lhat'$ (Definition \ref{def:coef-grad}) being complicated. In particular, for the nonnegativity of $\mu'$, some entries are nonnegative only through the sum of $\ltilde'$ and $\widetilde{H}^{-1}\lhat'$: for each $j$, the $(j-1, j)$ entry of $\ltilde'$ is negative and $(j-2, j)$ entry of $\widetilde{H}^{-1}\lhat'$ is negative. For details, see Appendix \ref{appendix:ogm-g}.
\end{proof}

We also derive the following efficient representation of the composite extension of OGM-G, which we call \emph{P-OGM-G} (proximal OGM-G). The derivation and analysis of this representation closely resemble those of POGM (Proposition \ref{prop:ogmeff}). P-OGM-G updates as (with $y_0 = z_0 = x_0$)
\begin{align*}
    y_{k+1} &= x_{k} - \nabla f(x_k), \\
    z_{k+1} &= y_{k+1} + \frac{(\theta_{n-k}-1)(2\theta_{n-k-1}-1)}{\theta_{n-k}(2\theta_{n-k}-1)}\left(y_{k+1}-y_k + \frac{z_k-x_k}{\alpha_{k, k-1}}\right) + \frac{2\theta_{n-k-1}-1}{2\theta_{n-k}-1}(y_{k+1}-x_k),\\
    x_{k+1} &= \prox_{\alpha_{k+1, k}h}(z_{k+1}),
\end{align*}
for $0 \leq k \leq n-1$ (recall that $\alpha_{k, k-1} = 1 + \frac{2\theta_{n-k}-1}{\theta_{n-k+1}}$; for $k = 0$, $\alpha_{k, k-1}$ is not used as $z_k - x_k = 0$), and outputs the final iterate $x_n$.

\begin{prop}[Efficient form as P-OGM-G] Fix $n \in \N$, and let $\{(\widetilde{x}_k, \widetilde{y}_k, \widetilde{z}_k): 0 \leq k \leq n\}$ be the iterates of P-OGM-G and $\{x_k: 0 \leq k \leq n\}$ be the iterates of the composite extension of OGM-G, with $\widetilde{x}_0 = x_0$. Then $\widetilde{x}_k = x_k$ for all $1 \leq k \leq n$. 
\end{prop}
\begin{proof}
    As in the proof of Proposition \ref{prop:ogmeff},
    by defining $x_0^{-} := x_0$ and $x_k^{-} := x_k + \alpha_{k, k-1}s_k$ for $1 \leq k \leq n$, we inductively show that $\widetilde{x}_k = x_k$ and $\widetilde{z}_k = x_k^{-}$. For $k = 0$ this holds by definition. For $k = 1$, we have
    \[\widetilde{z}_1 = \widetilde{y}_1 + \frac{(\theta_n-1)(2\theta_{n-1}-1)}{\theta_n(2\theta_n-1)}(\widetilde{y}_1 - \widetilde{y}_0) + \frac{2\theta_{n-1}-1}{2\theta_{n}-1}(\wt{y}_1 - \wt{x}_0) = \wt{x}_0 - \left(1+\frac{2\theta_{n-1}-1}{\theta_n}\right)\nabla f(\wt{x}_0) = x_0 - \alpha_{1, 0}g_0 = x_1^{-}\,,
    \]
    and thus $\widetilde{x}_1 = x_1$. Assume that the result holds for $1, \dots, k$. Then
    \begin{align*}
        \wt{z}_{k+1} &= \wt{y}_{k+1} + \frac{(\theta_{n-k}-1)(2\theta_{n-k-1}-1)}{\theta_{n-k}(2\theta_{n-k}-1)}\left(\wt{y}_{k+1}-\wt{y}_k + \frac{\wt{z}_k-\wt{x}_k}{\alpha_{k, k-1}}\right) + \frac{2\theta_{n-k-1}-1}{2\theta_{n-k}-1}(\wt{y}_{k+1}-\wt{x}_k) \\
        &= x_k - g_k + \frac{(\theta_{n-k}-1)(2\theta_{n-k-1}-1)}{\theta_{n-k}(2\theta_{n-k}-1)}(x_k - x_{k-1} - (g_k - g_{k-1}) + s_k) - \frac{2\theta_{n-k-1}-1}{2\theta_{n-k}-1}g_k \\
        &= x_k +  \frac{(\theta_{n-k}-1)(2\theta_{n-k-1}-1)}{\theta_{n-k}(2\theta_{n-k}-1)}(g_{k-1}+s_k) - \left(1+\frac{2\theta_{n-k-1}-1}{\theta_{n-k}}\right)g_k \\
        &\quad- \frac{(\theta_{n-k}-1)(2\theta_{n-k-1}-1)}{\theta_{n-k}(2\theta_{n-k}-1)}\sum_{j=0}^{k-1}\alpha_{k, j}(g_j + s_{j+1}) \\
        &= x_k - \sum_{j=0}^{k-2}\frac{(\theta_{n-k}-1)(2\theta_{n-k-1}-1)}{\theta_{n-k}(2\theta_{n-k}-1)}\alpha_{k, j}(g_j + s_{j+1}) \\
        &\quad- \frac{(\theta_{n-k}-1)(2\theta_{n-k-1}-1)}{\theta_{n-k}(2\theta_{n-k}-1)}(\alpha_{k, k-1}-1)(g_{k-1}+s_k) - \left(1+\frac{2\theta_{n-k-1}-1}{\theta_{n-k}}\right)g_k \\
        &= x_{k+1}^{-},
    \end{align*}
    where the first equality is from the induction hypothesis and the last equality is from \citep[equation 32]{ogm-g}. Thus, $\widetilde{x}_{k+1} = x_{k+1}$.
\end{proof}

\section{Discussion}\label{sec:discussion}
In this paper, we developed a general-purpose result for extending optimized first-order methods from the unconstrained setting to the composite setting, and we applied this to different combinations of algorithms and performance metrics. Our work suggests several directions for further inquiry, such as the following.

\paragraph*{Different classes of methods.} Our result connects methods in unconstrained and composite settings; can more general connections can be made? These can be, for example, between different problem settings (as in this paper), between different performance metrics (as in H-duality), or between more general algorithms (e.g., that use line search or adaptive updates). Developing these connections is a fundamental question in its own right and could help unify the design and analysis of algorithms.

\paragraph*{Further understanding of optimized methods.} The starting point for this work is the observation that optimized methods admit simple proof structures, namely the sum-of-squares term is of rank 0 or 1 (see the technical overview in Section \ref{subsec:overview}). This phenomenon holds beyond the optimized methods we covered here \citep[Chapter 5]{taylorhabilitation}, and for specific choices of stepsizes for GD it admits an equivalent characterization in terms of worst-case functions \citep[Proposition 4]{gswcomposing}. On the other hand, it is unclear whether proofs for ``non-optimized'' methods such as constant-stepsize GD or AGD admit such structures (see e.g., \citep{dt14, ogm}). It is an interesting open question to understand the generality of this phenomenon and whether it hints at an underlying unified theory for optimized methods.

\paragraph*{A full reduction.} While our main results (Theorems \ref{thm:main-func} and \ref{thm:main-grad}) provide reduction-style approaches that are much simpler to invoke than proving convergence rates from scratch, these reductions are not fully black-box as one needs to verify items (i) and (ii). 
It would be helpful for the design and analysis of algorithms if these reductions could be made fully black-box, or more generally 
if there are other unified approaches for solving the semidefinite programs arising from PEP-type analyses.

\section*{Acknowledgements} This work was partially supported by a Sloan Research Fellowship and a Seed Grant Award from Apple.

\addcontentsline{toc}{section}{References}
\bibliographystyle{plain}
\bibliography{ref}

\newpage
\appendix

\section{Proof of Theorem \ref{thm:main-func}}\label{appendix:proof-main-func}

\subsection{Implications of \eqref{eq:vanilla-func}}

In order to establish \eqref{eq:prox-func}, we first parse out \eqref{eq:vanilla-func} by comparing the coefficients on its each side.

\paragraph*{Linear form.} By comparing the coefficients of the linear form in $\{f_i\}$, we obtain
\begin{equation}\label{eq:vanilla-func-linear}
	\begin{split}
		\lambda_{i, \bullet} - \lambda_{\bullet, i} &= 0, \quad 0 \leq i \leq n-1 ,\\
		\lambda_{n, \bullet} - \lambda_{\bullet, n} &= -R_n, \\
		\lambda_{*, \bullet} &= R_n.
	\end{split}
\end{equation}

\paragraph*{Quadratic form.} Here, we compare the coefficients of the quadratic form in $x_0 - x_*, \{g_i\}$. The coefficient of $\angs{g_i}{g_j}$ for $0 \leq i < j \leq n$ on the left hand side of \eqref{eq:vanilla-func} is
\[\sum_{k=i+1}^n \wt{\alpha}_{k, i}\lambda_{k, j} -\wt{\alpha}_{j, i}\lambda_{\bullet, j}  + \sum_{k=j+1}^n \wt{\alpha}_{k, j}\lambda_{k, i} + \lambda_{i, j} + \lambda_{j, i} + \gamma_i \gamma_j = 0\,,
\]
and for $0 \leq i = j \leq n$ the coefficient is $\sum_{k=i+1}^n \wt{\alpha}_{k, i}\lambda_{k, i} - \frac{1}{2}(\lambda_{\bullet, i} + \lambda_{i, \bullet}) + \frac{1}{2}\gamma_i^2 = 0$.
This is equivalent to 
\begin{align}
	\lhat + \widetilde{H}\ltilde + (\widetilde{H}\ltilde)^T &= -\gtilde \gtilde^T, \label{eq:vanilla-func-quad-mat} \\
	\begin{bmatrix}
		\lambda_{0, n} + \lambda_{n, 0} \\
		\vdots \\
		\lambda_{n-1, n} + \lambda_{n, n-1}
	\end{bmatrix} + \widetilde{H}\begin{bmatrix}
		\lambda_{1, n} \\
		\vdots \\
		\lambda_{n-1, n} \\
		-\lambda_{\bullet, n}
	\end{bmatrix} &= -\gamma_n \gtilde, \label{eq:vanilla-func-quad-vec} \\
	\lambda_{\bullet, n} + \lambda_{n, \bullet} &= \gamma_n^2. \label{eq:vanilla-func-quad-scalar}
\end{align}
Also, by comparing the coefficient of $\angs{x_0 - x_*}{g_i}$ for $0 \leq i \leq n$, we obtain
\begin{equation}\label{eq:vanilla-func-quad-star}
	\lambda_{*, i} = \gamma_i, \quad 0 \leq i \leq n\,.
\end{equation}

\subsection{Verification of \eqref{eq:prox-func}}\label{subsec:verif-main-func}

Assuming \eqref{eq:vanilla-func} is true, we show that \eqref{eq:prox-func} is true by matching coefficients. It is clear that both sides of \eqref{eq:prox-func} are linear in $\{f_i\}, \{h_i\}$ and quadratic in $x_0 - x_*, \{g_i\}, \{s_i\}$.

\paragraph*{Linear form.} The coefficients of $\{f_i\}$ are from \eqref{eq:vanilla-func}. For the coefficients of $\{h_i\}$, it suffices to show that
\[\mtilde\bm{1}_n = -R_ne_n\,.
\]
From $\mtilde = \ltilde + \widetilde{H}^{-1}(\lhat - \gtilde \stilde^T)$ and $\ltilde \bm{1}_n = \begin{bmatrix}
	-\lambda_{1, n} \\ \vdots \\ -\lambda_{n-1, n} \\
	\lambda_{\bullet, n} - R_n
\end{bmatrix}$ from \eqref{eq:vanilla-func-linear}, it suffices to show that
\[\widetilde{H}^{-1}(\lhat - \gtilde \stilde^T) \bm {1}_n = \begin{bmatrix}
	\lambda_{1, n} \\ \vdots \\ \lambda_{n-1, n} \\
	-\lambda_{\bullet, n}
\end{bmatrix} \Leftrightarrow (\lhat - \gtilde \stilde^T) \bm {1}_n = \widetilde{H}\begin{bmatrix}
	\lambda_{1, n} \\ \vdots \\ \lambda_{n-1, n} \\
	-\lambda_{\bullet, n}
\end{bmatrix}\,.
\]
The right hand side is equal to the left hand side because
\begin{align*}
	- \begin{bmatrix}
		\lambda_{0, n} + \lambda_{n, 0} \\ \vdots \\ \lambda_{n-1, n} + \lambda_{n, n-1}
	\end{bmatrix}- \gamma_n\gtilde &= -\begin{bmatrix}
		\lambda_{0, n} + \lambda_{n, 0} \\ \vdots \\ \lambda_{n-1, n} + \lambda_{n, n-1}
	\end{bmatrix} - \gtilde -(\sum \stilde)\gtilde \\
	&= -\begin{bmatrix}
		\lambda_{0, n} + \lambda_{n, 0} + \lambda_{*, 0} \\ \vdots \\ \lambda_{n-1, n} + \lambda_{n, n-1} + \lambda_{*, n-1} 
	\end{bmatrix} -(\sum \stilde)\gtilde \\
	&= (\lhat - \gtilde \stilde^T)\bm{1}_n\,,
\end{align*}
where the first equality is from 
\begin{equation}\label{eq:sumsigma}
	\sum \sigma = \sum \stilde + \sigma_*  = \frac{\lambda_{\bullet, n} + \lambda_{n, \bullet}}{\gamma_n} = \frac{\gamma_n^2}{\gamma_n} = \gamma_n\,,
\end{equation}
by \eqref{eq:vanilla-func-quad-scalar} and the second equality is from \eqref{eq:vanilla-func-quad-star}.

\paragraph*{Quadratic form.} Now, we compare the coefficients of the quadratic forms (in $x_0 - x_*, \{g_i\}, \{s_i\}$). The analysis is more technically involved than that for the linear form, because $Q_{ij}$ (in unconstrained setting) is not equal to $Q_{ij}^f$ (in composite setting). Thus, we define and analyze an intermediate quantity that connects these quantities. After this, the comparison is conceptually straightforward yet algebraically tedious.

\paragraph*{Setup.} 
It is convenient to first characterize $\sum_{i, j}\lambda_{i, j}Q^f_{ij}$, from \eqref{eq:vanilla-func} (which is written in $Q_{ij}$, for unconstrained setting). For this, define the \emph{pseudo-co-coercivities} $\widetilde{Q}^f_{ij}$ obtained by replacing $(x_0-x_*, g_0,  g_1, \dots, g_n)$ from $Q_{ij}$ with $(x_0 - x_*, g_0 + s_1, g_1+s_2, \dots, g_{n-1}+s_n, g_n+ s_{n+1})$, with $s_{n+1} := 0$. Then \eqref{eq:vanilla-func} implies
\begin{equation}\label{eq:pseudococo}
	\sum_{i, j} \lambda_{i, j}\widetilde{Q}^f_{ij} + \frac{1}{2}\norm{x_0 - x_* - \begin{bmatrix}
			g_0 + s_1 | g_1 + s_2 | \dots | g_n + s_{n+1}
		\end{bmatrix} \gamma}^2 = R_n(f_* - f_n) + \frac{1}{2}\norm{x_0 - x_*}^2\,.
\end{equation}
The difference between the actual co-coercivities $Q_{ij}^f$ and pseudo-co-coercivities $\widetilde{Q}_{ij}^f$ satisfies
\begin{equation}\label{eq:pseudococodiff}
	\begin{split}
		Q^f_{ij} - \widetilde{Q}^f_{ij} &= \angs{s_{j+1}}{x_i - x_j} - \frac{1}{2}\norm{g_i - g_j}^2 + \frac{1}{2}\norm{g_i + s_{i+1} - g_j - s_{j+1}}^2 \\
		&= \angs{s_{j+1}}{x_i - x_j} +  \frac{1}{2}\norm{s_{i+1} - s_{j+1}}^2 + \angs{g_i - g_j}{s_{i+1}-s_{j+1}}\,,
	\end{split}
\end{equation}
with $s_{*+1} := s_*$. Therefore, by adding $\sum_{i, j}\lambda_{i, j}(Q_{ij}^f - \wt{Q}_{ij}^f)$ on both sides of \eqref{eq:pseudococo} using \eqref{eq:pseudococodiff},
\begin{align*}
	&\sum_{i, j} \lambda_{i, j}Q^f_{ij} + \frac{1}{2}\norm{x_0 - x_* - \begin{bmatrix}
			g_0 + s_1 | g_1 + s_2 | \dots | g_n + s_{n+1}
		\end{bmatrix} \gamma}^2 \\
	&= R_n(f_* - f_n) + \frac{1}{2}\norm{x_0 - x_*}^2 + \sum_{i \neq *, j} \lambda_{i, j} \left(\angs{s_{j+1}}{x_i - x_j} + \frac{1}{2}\norm{s_{i+1} - s_{j+1}}^2 + \angs{g_i - g_j}{s_{i+1} - s_{j+1}}\right) \\
	&\quad+ \sum_j \lambda_{*j} \left(\angs{s_{j+1}}{x_* - x_j} + \frac{1}{2}\norm{s_{j+1}}^2 - \frac{1}{2}\norm{s_*}^2 + \angs{g_j}{s_{j+1} - s_*}\right)\,,
\end{align*}
where we split the sum $\sum_{i,j} \lambda_{i, j}(Q_{ij}-\widetilde{Q}_{ij}^f)$ based on $i \neq *$ and $i = *$. Thus, \eqref{eq:prox-func} is equivalent to
\begin{equation}\label{eq:frompseudococo-func}
	\begin{split}
		\sum_{i, j} \mu_{i, j}Q^h_{ij} 
		&= R_n(h_* - h_n) + \frac{1}{2}\xi \norm{x_0 - x_*}^2 \\
		&\quad + \angs{x_0 - x_* - \begin{bmatrix}
				g_0 + s_1 | g_1 + s_2 | \dots | g_n + s_{n+1}
			\end{bmatrix} \gamma}{\begin{bmatrix}
				s_1 | s_2 | \dots | s_n | s_*
			\end{bmatrix} \sigma} - \frac{1}{2}\norm{\begin{bmatrix}
				s_1 | s_2 | \dots | s_n | s_*
			\end{bmatrix} \sigma}^2 \\
		&\quad- \sum_{i \neq *, j} \lambda_{i, j} \left(\angs{s_{j+1}}{x_i - x_j} + \frac{1}{2}\norm{s_{i+1} - s_{j+1}}^2 + \angs{g_i - g_j}{s_{i+1} - s_{j+1}}\right) \\
		&\quad- \sum_j \lambda_{*j} \left(\angs{s_{j+1}}{x_* - x_j} + \frac{1}{2}\norm{s_{j+1}}^2 - \frac{1}{2}\norm{s_*}^2 + \angs{g_j}{s_{j+1} - s_*}\right) - \frac{1}{2}\Tr(VSV^T)\,.
	\end{split}
\end{equation}

\paragraph*{Category of vectors.} For convenience, we categorize the vectors (that constitute the quadratic form) into different types as follows:
\begin{equation}\label{eq:vectortype}
	\begin{split}
		g &: g_0, \dots, g_n \text{ (\underline{g}radients)}, \\
		s &: s_1, \dots, s_n \text{ (\underline{s}ubgradients)},\\
		i &: x_0 - x_* \text{ (\underline{i}nitial distance)}, \\
		o &: s_* = -g_* \text{ (\underline{o}ptimum)}.
	\end{split}
\end{equation}
In our subsequent analysis, we divide into different cases based on the combination of these types.

\paragraph*{Coefficients of $(g, g)$.} These are directly from \eqref{eq:vanilla-func}.

\paragraph*{Coefficients of $(g, s)$.}

The coefficient of $\angs{g_i}{s_j}$ on the left hand side of \eqref{eq:frompseudococo-func} for $0 \leq i \leq n-1, 1 \leq j \leq n$ is
\[\sum_{k=i+1}^n \wt{\alpha}_{k, i} \mu_{k, j}  -\wt{\alpha}_{j, i} \mu_{\bullet, j}  \bm{1}\{i \leq j-1\}\,,
\]
whereas the corresponding coefficient on the right hand side of \eqref{eq:frompseudococo-func} is
\[-\gamma_i \sigma_j + \sum_{k=i+1}^n \wt{\alpha}_{k, i}\lambda_{k, j-1} - \wt{\alpha}_{j-1, i}\lambda_{\bullet, j-1}\bm{1}\{i \leq j-2\} + \lambda_{i, j-1} + \lambda_{j-1, i} - (\lambda_{\bullet, j-1} + \lambda_{j-1, \bullet})\bm{1}\{i = j-1\}\,.
\]
Thus, the coefficients being matched is equivalent to (with a correspondence of entry $(i, j) \Leftrightarrow \angs{g_{i-1}}{s_j}$)
\[\wt{H}\mtilde = -\gtilde \stilde^T + \wt{H}\ltilde + \lhat\,,
\]
which is true from \eqref{eq:mutilde-func}. Note that the alternative form in \eqref{eq:mutilde-func} is obtained from \eqref{eq:vanilla-func-quad-mat}.

The coefficient of $\angs{g_n}{s_j}$ for $1 \leq j \leq n$ on the left hand side of \eqref{eq:frompseudococo-func} is 0. For the right hand side it is $-\gamma_n \sigma_j + \lambda_{j-1, n} + \lambda_{n, j-1} = 0$ from \eqref{eq:sigma}.

\paragraph*{Coefficients of $(s, s)$.} The coefficient of $\angs{s_i}{s_j}$ for $1 \leq i \leq j \leq n$ on the left hand side of \eqref{eq:frompseudococo-func} is
\[\begin{cases}
	\sum_{k=i}^n \wt{\alpha}_{k, i-1}\mu_{k, j} - \wt{\alpha}_{j, i-1}\mu_{\bullet, j} + \sum_{k=j}^n \wt{\alpha}_{k, j-1}\mu_{k, i} & j > i \\
	\sum_{k=i}^n \wt{\alpha}_{k, i-1}\mu_{k, i} - \wt{\alpha}_{i, i-1}\mu_{\bullet, i} & j = i\,,
\end{cases}
\]
and the corresponding coefficient on the right hand side is
\begin{align*}
	&-\gamma_{i-1}\sigma_j - \gamma_{j-1}\sigma_i - \sigma_i \sigma_j + \sum_{k=i}^n \wt{\alpha}_{k, i-1} \lambda_{k, j-1} - \wt{\alpha}_{j-1, i-1}\lambda_{\bullet, j-1} + \sum_{k=j}^n \wt{\alpha}_{k, j-1}\lambda_{k, i-1} + \lambda_{i-1, j-1} + \lambda_{j-1, i-1} \\
	&\quad+ \lambda_{i-1, j-1} + \lambda_{j-1, i-1} + \sigma_i \sigma_j \\
	&= -\gamma_{i-1}\sigma_j - \gamma_{j-1}\sigma_i + \sum_{k=i}^n \wt{\alpha}_{k, i-1} \lambda_{k, j-1}  - \wt{\alpha}_{j-1, i-1}\lambda_{\bullet, j-1} + \sum_{k=j}^n \wt{\alpha}_{k, j-1}\lambda_{k, i-1} + 2(\lambda_{i-1, j-1} + \lambda_{j-1, i-1} )
\end{align*}
if $i < j$ and 
\begin{align*}
	&-\gamma_{i-1}\sigma_i - \frac{1}{2}\sigma_i^2 + \sum_{k=i}^n \wt{\alpha}_{k, i-1}\lambda_{k, i-1} - \mu_{\bullet, i}\wt{\alpha}_{i, i-1} - \frac{1}{2}(\lambda_{i-1, \bullet} + \lambda_{\bullet, i-1}) -\frac{1}{2}(\lambda_{i-1, \bullet} + \lambda_{\bullet, i-1}) + \frac{1}{2}\sigma_i^2 \\
	&= -\gamma_{i-1}\sigma_i + \sum_{k=i}^n \wt{\alpha}_{k, i-1}\lambda_{k, i-1} - \mu_{\bullet, i}\wt{\alpha}_{i, i-1} - (\lambda_{i-1, \bullet} + \lambda_{\bullet, i-1})
\end{align*}
if $i = j$. Thus the coefficients matching is equivalent to
\[\wt{H}\mtilde + (\wt{H}\mtilde)^T = -(\gtilde \stilde^T + \stilde \gtilde^T) + \wt{H}\ltilde + (\wt{H}\ltilde)^T + 2\lhat\,,
\]
which holds from \eqref{eq:mutilde-func} as $\wt{H}\mtilde = \wt{H}\ltilde + \lhat - \gtilde \stilde^T$.

\paragraph*{Coefficients of remaining combinations.} The rest of the combinations for inner products involves $x_0 - x_*$ or $s_* = -g_*$, which only appears in a few entries of \eqref{eq:frompseudococo-func}.

\begin{itemize}
	\item $(g, o), (s, o), (o, o)$: In \eqref{eq:frompseudococo-func}, the left hand side has coefficients $0$, and the right hand side has coefficients (respectively for $\angs{g_i}{s_*}, \angs{s_j}{s_*}, \norm{s_*}^2$) $-\gamma_i \sigma_* + \lambda_{*, i} = 0, -\gamma_{j-1}\sigma_* - \sigma_j\sigma_* + \gamma_{j-1}\sigma_* + \sigma_j \sigma_* = 0, -\frac{1}{2}\sigma_*^2 + \frac{1}{2}\sigma_*^2 = 0$.
	\item $(i, i), (g, i)$: Straightforward from \eqref{eq:vanilla-func}.
	\item $(s, i), (o, i)$: In \eqref{eq:frompseudococo-func}, for $\angs{s_j}{x_0 - x_*}$ the left hand side has coefficients $\mu_{*, j}$ and the right hand side has coefficients $\sigma_j + \lambda_{*, j-1} - v_j = \mu_{*, j}$.
\end{itemize}

\subsection{Other properties}
In this section, we conclude the proof of Theorem \ref{thm:main-func} by providing the remaining details.
\paragraph*{$\boldsymbol{\sum v = 0}$.} From the definition of $v$ in \eqref{eq:slack-func},
\begin{align*}
	\sum v &= \sum_{k=1}^n(\sigma_k + \lambda_{*, k-1} - \mu_{*, k}) + \sigma_* \\
	&= \sum\sigma  + (\lambda_{*, \bullet} - \lambda_{*, n}) - \mu_{*, \bullet}\,,
\end{align*}
where $\sum \sigma = \lambda_{*, n} = \gamma_n$ from \eqref{eq:vanilla-func-quad-star} and \eqref{eq:sumsigma}. Thus $\sum v = R_n - \mu_{*, \bullet}$.
The term $\mu_{*, \bullet}$ can be calculated as follows:
\begin{align*}
	\mu_{*, \bullet} &= -\bm{1}_n^T \mtilde \bm{1}_n \\
	&= -\bm{1}_n^T \wt{H}^{-1}\left(-\begin{bmatrix}
		\lambda_{n, 0} + \lambda_{*, 0} + \lambda_{0, n} \\
		\vdots \\
		\lambda_{n, n-1} + \lambda_{*, n-1} + \lambda_{n-1, n}
	\end{bmatrix}  - \left(\sum \stilde\right)\gtilde\right) - \bm{1}_n^T \begin{bmatrix}
		\lambda_{1, \bullet} - \lambda_{1, n}- \lambda_{\bullet, 1} \\
		\vdots \\
		\lambda_{n-1, \bullet} - \lambda_{n-1, n} - \lambda_{\bullet, n-1} \\
		\lambda_{n, \bullet}
	\end{bmatrix} \\
	&= \bm{1}_n^T \wt{H}^{-1}(\gtilde + \stilde)\gamma_n + \lambda_{\bullet, n} - \lambda_{*, n} - \lambda_{n, \bullet} \\
	&= -\bm{1}_n^T \begin{bmatrix}
		\lambda_{1, n} \\
		\vdots \\
		-\lambda_{\bullet, n}
	\end{bmatrix} - \lambda_{*, n} + R_n \\
	&= \lambda_{*, n} - \lambda_{*, n} + R_n = R_n\,,
\end{align*}
where the second equality is from the definition of $\mtilde$ \eqref{eq:mutilde-func}, the third equality is from \eqref{eq:vanilla-func-linear} and \eqref{eq:sumsigma}, and the fourth equality is from \eqref{eq:vanilla-func-quad-vec}. Thus, $\sum v = 0$.

\paragraph*{$L$ is Laplacian.} Recall from \eqref{eq:slack-func} that $L = -\begin{bmatrix}
	\lhat & \gtilde \\
	\gtilde^T & -\lambda_{*, \bullet}
\end{bmatrix} - \sigma \sigma^T$. Here, the nondiagonal entries are clearly nonpositive. The $j$th row sum for $1 \leq j \leq n$ is
\[\lambda_{j-1, n} + \lambda_{n, j-1} + \lambda_{*, j-1} - \gamma_{j-1} - \left(\sum \sigma\right)\sigma_j = \sigma_j\left(\gamma_n - \sum \sigma\right) = 0\,,
\]
where the first equality is from the definition of $\sigma$ \eqref{eq:sigma}, $\lambda_{*, j-1} = \gamma_{j-1}$ from \eqref{eq:vanilla-func-quad-star} and the last equality is from \eqref{eq:sumsigma}. The $(n+1)$th row sum is $-\sum \gtilde + \lambda_{*, \bullet} - \left(\sum \sigma\right)\sigma_* = \lambda_{*, \bullet} - \sum \gamma = 0$ from \eqref{eq:vanilla-func-quad-star}.

\section{Proof of Theorem \ref{thm:main-grad}}\label{appendix:proof-main-grad}
As in Appendix \ref{appendix:proof-main-func}, we show that \eqref{eq:prox-grad} is true by matching coefficients (whose both sides are linear in $\{f_i\}, \{h_i\}$ and quadratic in $x_0 - x_*, \{g_i\}, \{s_i\}$), assuming \eqref{eq:vanilla-grad} is true.

\subsection{Implications of \eqref{eq:vanilla-grad}}
We record properties of $\lambda'$ by comparing the coefficients in \eqref{eq:vanilla-grad}.

\paragraph*{Linear form.} The coefficients of the linear form in $\{f_i\}$ yield
\begin{equation}\label{eq:vanilla-grad-linear}
	\begin{split}
		\lambda'_{0, \bullet} - \lambda'_{\bullet, 0} &= 1,\\
		\lambda'_{i, \bullet} - \lambda'_{\bullet, i} &= 0\,,\quad i \in \{1, \dots, n-1, *\}, \\
		\lambda'_{n, \bullet} - \lambda'_{\bullet, n} &= -1.
	\end{split}
\end{equation}

\paragraph*{Quadratic form.} The coefficients of the quadratic form in $x_0 - x_*, \{g_i\}$ yield
\begin{align}
	\lhat' + \widetilde{H}\ltilde' + (\widetilde{H}\ltilde')^T &= \bm{0}_{n \times n} , \label{eq:vanilla-grad-quad-mat} \\
	\begin{bmatrix}
		\lambda'_{0, n} + \lambda'_{n, 0} \\
		\vdots \\
		\lambda'_{n-1, n} + \lambda'_{n, n-1}
	\end{bmatrix} + \widetilde{H}\begin{bmatrix}
		\lambda'_{1, n} \\
		\vdots \\
		-\lambda'_{\bullet, n}
	\end{bmatrix} &= \bm{0}_n , \label{eq:vanilla-grad-quad-vec} \\
	\lambda'_{\bullet, n} + \lambda'_{n, \bullet} &= R'_n . \label{eq:vanilla-grad-quad-scalar}
\end{align}

\subsection{Verification of \eqref{eq:prox-grad} and positive semidefiniteness}

Now, we compare the coefficients in \eqref{eq:prox-grad}. The derivation for each case is very similar to that in Appendix \ref{subsec:verif-main-func} and thus we focus on the equivalent forms presented in matrix.

\paragraph*{Linear form.} The coefficients for $\{f_i\}$ are from \eqref{eq:vanilla-grad}. For the coefficients for $\{h_i\}$, it suffices to show that
\[\mtilde'\bm{1}_n = -e_n\,.
\]
From $\mtilde' = \ltilde' + \widetilde{H}^{-1}\lhat'$ and $\ltilde' \bm{1} = \begin{bmatrix}
	-\lambda'_{1, n} \\ \vdots \\ -\lambda'_{n-1, n} \\
	\lambda'_{\bullet, n} - 1
\end{bmatrix}$ (from \eqref{eq:vanilla-grad-linear}), it suffices to show that
\[\widetilde{H}^{-1}\lhat' \bm {1}_n = \begin{bmatrix}
	\lambda'_{1, n} \\ \vdots \\ \lambda'_{n-1, n} \\
	-\lambda'_{\bullet, n}
\end{bmatrix} \Leftrightarrow \lhat' \bm{1}_n = \widetilde{H}\begin{bmatrix}
	\lambda'_{1, n} \\ \vdots \\ \lambda'_{n-1, n} \\
	-\lambda'_{\bullet, n}
\end{bmatrix}\,,
\]
which holds from \eqref{eq:vanilla-grad-quad-vec}.

\paragraph*{Quadratic form.} By using the pseudo-co-coercivities defined in Appendix \ref{subsec:verif-main-func}, \eqref{eq:prox-grad} given \eqref{eq:vanilla-grad} is equivalent to
\begin{equation}\label{eq:frompseudococo-grad}
	\begin{split}
		\sum_{i, j} \mu'_{i, j}Q^h_{ij} 
		&= \frac{R'_n}{2}\norm{g_n}^2 - \frac{R'_n}{2}(1-\xi')\norm{g_n + s_n}^2 + h_0 - h_* \\
		&\quad- \sum_{i, j} \lambda'_{i, j} \left(\angs{s_{j+1}}{x_i - x_j} + \frac{1}{2}\norm{s_{i+1} - s_{j+1}}^2 + \angs{g_i - g_j}{s_{i+1} - s_{j+1}}\right) \\
		&\quad- \frac{1}{2}\Tr(V'S'(V')^T)\,.
	\end{split}
\end{equation}

Note that the term $- \frac{R'_n}{2}(1-\xi')\norm{g_n + s_n}^2$ in \eqref{eq:frompseudococo-grad} is cancelled by the summand $-R'_n(1-\xi')(e_1+e_{n+1})(e_1 + e_{n+1})^T$ in $S'$ from \eqref{eq:slack-grad}. Also, for the types of vectors defined in \eqref{eq:vectortype}, we only need to consider type $g$ and $s$ as we do not use the index $*$.
\paragraph*{Coefficients of $(g, g)$.} These are directly from \eqref{eq:vanilla-grad}.

\paragraph*{Coefficients of $(g, s)$.} The coefficients of $\angs{g_i}{s_j}$ for $0 \leq i \leq n-1, 1 \leq j \leq n$ being matched is equivalent to
\[\wt{H}\mtilde' = \wt{H}\ltilde' + \lhat'\,,
\]
which holds from \eqref{eq:vanilla-grad-quad-vec} and $\wt{H}\mtilde' = -(\wt{H}\ltilde)^T$ from \eqref{eq:mutilde-grad}.

The coefficient of $\angs{g_n}{s_j}$ for $1 \leq j \leq n$ on the left hand side of \eqref{eq:frompseudococo-grad} is 0; for the right hand side it is $\lambda_{j-1, n} + \lambda_{n, j-1} - v'_j = 0$ from \eqref{eq:slack-grad}.

\paragraph*{Coefficients of $(s, s)$.} By considering the coefficients of $\angs{s_i}{s_j}$ for $1 \leq i \leq j \leq n$, it suffices to show
\[\wt{H}\mtilde' + (\wt{H}\mtilde')^T = \wt{H}\ltilde' + (\wt{H}\ltilde')^T + 2\lhat'\,,
\]
which follows from the corresponding result for the coefficients of $(g, s)$.

\paragraph*{Structure of $S'$ and diagonal dominance.} Recall the definition of $S'$ from \eqref{eq:slack-grad}:
\[S' = \begin{bmatrix}
	R_n' & (v')^T \\
	v' & -\lhat'
\end{bmatrix} - R'_n(1-\xi')(e_1 + e_{n+1})(e_1 + e_{n+1})^T\,,
\]
where $v' = [v'_1, \dots, v'_n]$ with $ v'_i = \lambda_{i-1, n}' + \lambda_{n, i-1}'$. 

The first summand $\begin{bmatrix}
	R_n' & (v')^T \\
	v' & -\lhat'
\end{bmatrix}$ is diagonally dominant: in the first row, $v'_i \geq 0$ with $R'_n = \lambda'_{\bullet, n} + \lambda'_{n, \bullet} = \sum v'$ from \eqref{eq:vanilla-grad-quad-scalar}; in the $i$th row for $2 \leq i \leq n+1$, the diagonal entry is $\lambda'_{\bullet, i-1} + \lambda'_{i-1, \bullet}$, and the absolute sum of the remaining entries is $v'_{i-1} + \sum_{k \notin \{i-1, n\}}(\lambda'_{k, i-1} + \lambda'_{i-1, k}) = \lambda'_{\bullet, i-1} + \lambda'_{i-1, \bullet}$.

Thus, if we choose $\xi' \in [0, 1)$ such that the $(1, n+1)$ entry of $S'$ is nonnegative, i.e., $\lambda'_{n-1, n} + \lambda'_{n, n-1} - R'_n(1-\xi') \geq 0$, $S'$ is diagonally dominant. In particular, we can choose $\xi' = 1 - \frac{\lambda'_{n-1, n} + \lambda'_{n, n-1}}{R'_n}$ (which is always in $[0, 1]$ from \eqref{eq:vanilla-grad-quad-scalar}).

\section{Deferred details for Section \ref{sec:app-grad}}\label{appendix:appgrad}

\subsection{Stepsize-accelerated proximal GD}\label{appendix:sss-g}

\begin{proof}[Proof of Theorem \ref{thm:sss-g}]
	As in the case of objective function minimization, we first show that the following key lemma holds.
	\begin{lemma}\label{lem:sss-gpartialsum}
		For $k \in \N$ and $n = 2^k-1$, let $\lambda' = \lambda'^{(k)}$. Then
		\begin{equation}\label{eq:lambdascaled-grad}
			\frac{\lambda'_{i, j-1}}{\gamma'_{j-1}} - \frac{\lambda'_{i, j}}{\gamma'_j} \begin{cases}
				\geq 0 & 0 \leq i \leq j-2 \\
				\leq 0 & i \geq j+1
			\end{cases}
		\end{equation}
		for all $1 \leq j \leq n$, where $[\gamma'_0, \dots, \gamma'_{n-1}] := w^{(k)}$ and $\gamma'_n := 1$.
	\end{lemma}
	\begin{proof}
		For $k = 2$ (note that no such entries are included for $k = 1$), all inequalities are straightforward (as either $\lambda'_{i, j-1}$ or $\lambda'_{i, j}$ is 0) except $\frac{\lambda'_{3, 1}}{\gamma'_1}-\frac{\lambda'_{3, 2}}{\gamma'_2} = \frac{\sqrt{2}}{1+\sqrt{2}} - \frac{2+\sqrt{2}}{\sqrt{2}} \leq 0$ and $\frac{\lambda'_{1, 2}}{\gamma'_2} - \frac{\lambda'_{1, 3}}{\gamma'_3} = \frac{\sqrt{2}}{\sqrt{2}} -1 = 0$.
		
		Assume that the result hold for $k$, and consider $k+1$. For $1 \leq j \leq n$, \eqref{eq:lambdascaled-grad} holds by the induction hypothesis with $w^{(k+1)} = [w^{(k)}, \eta_k, \pi^{(k)}]$.
		
		For $j = n+1$, we have $\lambda'_{i, j-1} = 0$ for all $n+1 \leq i \leq 2n$ and $\lambda'_{i, j} = 0$ for all $0 \leq i \leq n-1$. Thus, it only suffices to consider $i = 2n+1$ where 
		\[\frac{\lambda'_{2n+1, n}}{\gamma'_n} - \frac{\lambda'_{2n+1, n+1}}{\gamma'_{n+1}} = \frac{1}{\eta_k}\left(\frac{\tau_{k+1}}{2\rho^k}-1\right) - \frac{1}{\sqrt{2}}\frac{\tau_{k+1}}{2\rho^{2k}}\sqrt{2} \leq 0 \ \Leftrightarrow \ \eta_k \geq \rho^k\left(1-\frac{2\rho^k}{\tau_{k+1}}\right) \,,
		\]
		which holds from $\eta_k = 1+\frac{\sqrt{\tau_k(\tau_k+8\rho^k)}-\tau_k}{4}=1+\rho^k -\frac{4\rho^k}{\sqrt{\tau_k(\tau_k+8\rho^k)}+\tau_k + 4\rho^k} = 1+\rho^k\left(1-\frac{2\rho^k}{\tau_{k+1}}\right)$.
		
		For $n+2 \leq j \leq 2n+1$, \eqref{eq:lambdascaled-grad} holds from Lemma \ref{lem:ssspartialsum}.
	\end{proof}
	
	By Lemma \ref{lem:sss-gpartialsum}, for $1 \leq j < i \leq n-1$,
	\[\mu'_{i, j} = \gamma'_{j-1}\left(\frac{\sum_{l=0}^{j-1}\lambda'_{l, i-1}}{\gamma'_{i-1}} - \frac{\sum_{l=0}^{j-1} \lambda'_{l, i}}{\gamma'_i}\right) \geq 0\,,
	\]
	and for $1 \leq i < j \leq n$,
	\[\mu'_{i, j} = \gamma'_{j-1}\left(\frac{\sum_{l=j}^n \lambda'_{l, i}}{\gamma'_i} - \frac{\sum_{l=j}^n \lambda'_{l, i-1}}{\gamma'_{i-1}}\right) \geq 0\,.
	\]
	For $i = n$ and $1 \leq j \leq n-1$,
	\[\mu'_{n, j} = \frac{\gamma'_{j-1}}{\gamma'_{n-1}}\sum_{l=0}^{j-1}\lambda'_{l, n-1} \geq 0\,,
	\]
	and for $i = 0$ and $1 \leq j \leq n$, $\begin{bmatrix}
		\mu'_{0, 1} & \dots & \mu'_{0, n}
	\end{bmatrix} = -\bm{1}_n^T \mtilde = e_1$.
	
	For item (ii), as in Theorem \ref{thm:main-grad} we take $\xi' = 1 - \frac{\lambda'_{n, n-1} + \lambda'_{n-1, n}}{R'_n}$. For $k = 1$ (recall $n = 2^k-1$) we have $\lambda'_{n, n-1} + \lambda'_{n-1, n} = 3$; for $k \geq 2$,
	\begin{align*}
		\lambda'_{n, n-1} + \lambda'_{n-1, n} &= \frac{\tau_k}{\rho^{2(k-1)}}(\blambda'^{(k-1)}_{(n-1)/2, (n-3)/2} + \blambda'^{(k-1)}_{(n-3)/2, (n-1)/2}) + \frac{\tau_k}{2\rho^{2(k-1)}}\sqrt{2} \\
		&= \frac{\tau_k}{\rho^{2(k-1)}}\left(\frac{1}{\sqrt{2}}(\rho^{2k-3}-1) + \rho^{2k-3}\right) + \frac{\tau_k}{2\rho^{2(k-1)}}\sqrt{2} \\
		&= \frac{\tau_k}{\sqrt{2}}\,,
	\end{align*}
	where the formula for $a_k := \blambda'^{(k-1)}_{(n-1)/2, (n-3)/2}, b_k := \blambda'^{(k-1)}_{(n-3)/2, (n-1)/2} (
	k \geq 2)$ are obtained from respectively solving $a_2 = 1, a_{l+1} = \rho^2 a_l + \rho\sqrt{2}$ and $b_2 = \rho, b_{l+1} = \rho^2 b_l$ from the recursive definition in Proposition \ref{prop:gswknown}. Finally, from \citep[Lemma 4]{gsw24}, $\frac{1}{\tau_k} \sim \frac{1}{(\rho-1)(1+1/\sqrt{\rho})\rho^k} \sim \frac{\rho-\sqrt{\rho}}{2n^{\log_2 \rho}}$.
\end{proof}

\subsection{Proximal OGM-G}\label{appendix:ogm-g}
Before the proof of Theorem \ref{thm:ogm-g}, we present the following lemma that will be used throughout our analysis. This quantitatively characterizes the sequences $\{\theta_i\}$ and $\{\varphi_i\}$.

\begin{lemma}\label{lem:thetaphi}
	\begin{itemize}
		\item[(a)] (Difference of $\{\theta_i\}$). $\theta_{i+1} - \theta_i > \frac{1}{2}$ for all $0 \leq i \leq n-1$ and $\theta_{i+1}-\theta_i < \frac{3}{4}$ for all $0 \leq i \leq n-2$.
		\item[(b)] (Concentration of $\{\varphi_i\}$). $\varphi_{i+1} \geq \frac{4}{3}, \varphi_{i+2} \geq \frac{15}{11}, \varphi_{i+3} \geq \frac{7}{5}$ for all $i \geq 1$ and $\varphi_i \leq \frac{3}{2}, \varphi_{i+1} \leq 1+\frac{1}{\sqrt{2}}$ for all $1 \leq i \leq n-1$.
	\end{itemize}
\end{lemma}
\begin{proof}
	\begin{itemize}
		\item[(a)] For $0 \leq i \leq n-2$, $\theta_i + \frac{1}{2} < \theta_{i+1} = \frac{1 + \sqrt{1 + 4\theta_i^2}}{2} < \frac{1 + \sqrt{4\theta_i^2 + 2\theta_i + \frac{1}{4}}}{2} = \theta_i + \frac{3}{4}$. $\theta_n - \theta_{n-1} > \frac{1}{2}$ directly follows.
		\item[(b)] The upper bounds hold from $\theta_{i-1} \leq \theta_i \Rightarrow \varphi_i = 1+\frac{\theta_{i-1}}{2\theta_i} \leq \frac{3}{2}$ for all $1 \leq i \leq n-1$ and $\sqrt{2}\theta_{n-1} \leq \theta_n \Rightarrow \varphi_n = 1 + \frac{\theta_{n-1}}{\theta_n} \leq 1+\frac{1}{\sqrt{2}}$.
		For the lower bounds, it is straightforward to see that $\varphi_i$ is increasing in $i$ and $\varphi_n \geq \frac{3}{2}$. Thus, it suffices to show that $\varphi_2 \geq \frac{4}{3}, \varphi_3 \geq \frac{15}{11}, \varphi_4 \geq \frac{7}{5}$ (for $n \geq 5$).
		
		These are respectively equivalent to 
		\[\theta_2 \leq \frac{3}{2}\theta_1, \quad \theta_3 \leq \frac{11}{8}\theta_2, \quad \theta_4 \leq \frac{5}{4}\theta_3\,.\]
		From (a), we know that $\theta_i \geq \frac{i}{2}+1$ and in particular $\theta_1 \geq \frac{3}{2}, \theta_2 \geq 2$. Thus from (a), $\theta_2 \leq \theta_1 + \frac{3}{4} \leq \frac{3}{2}\theta_1$ and $\theta_3 \leq \theta_2 + \frac{3}{4} \leq \frac{11}{8}\theta_2$, proving the first two inequalities. The last inequality holds from $\theta_4 = \frac{1 + \sqrt{1 + 4\theta_3^2}}{2} \leq \frac{1 + \frac{5}{2}\theta_3-1}{2} = \frac{5}{4}\theta_3$, from $\theta_3 \geq \frac{5}{2}$.
	\end{itemize}
\end{proof}

\begin{proof}[Proof of Theorem \ref{thm:ogm-g}]
	For item (i), we calculate $\mtilde' = \ltilde' + \widetilde{H}^{-1}\lhat'$ and show the corresponding nonnegativity. For notational convenience, we will mainly consider these matrices scaled by $1/\theta_n^2$.

	First, note that from
	\[[(1/\theta_n^2)\ltilde']_{i, j} = \begin{cases}
		-\frac{1}{2\theta_{n-i-1}^2} & j = i+1 \\
		\frac{1}{2\theta_{n-i-1}^2} & j = i + 2 \\
		\frac{1}{2\theta_{n-1}^2} - \frac{1}{\theta_n^2} & i = n, j = 1 \\
		\frac{1}{2\theta_{n-j}^2} - \frac{1}{2\theta_{n-j+1}^2} & i = n, 2 \leq j \leq n \\
		0 & \text{else}\,,
	\end{cases}\]
	its $(j-1, j)$ entries are negative, whereas the rest are nonnegative.
	
	The main technical part is to calculate $\widetilde{H}^{-1}\lhat'$. First, note that the stepsize matrix $H$ of OGM-G satisfies the following (the proof is analogous to that of Lemma \ref{lem:ogmHmatrix} and hence is omitted):
	\begin{gather*}
		H = U(\theta_n, \dots, \theta_1)^{-1}U(\varphi_n, \dots, \varphi_1)\diag(2\theta_{n-1}, \dots, 2\theta_0) \\
		\Rightarrow \widetilde{H}^{-1} = U_n^{-1}\diag(1/(2\theta_{n-1}), \dots, 1/(2\theta_0))U(\varphi_n, \dots, \varphi_1)^{-1}U(\theta_n, \dots, \theta_1)\,.
	\end{gather*}
	Thus with $[(1/\theta_n^2)\lhat']_{i, j} = \begin{cases}
		\frac{1}{2\theta_{n-i}^2} & i = j-1 \\
		\frac{1}{2\theta_{n-j}^2} & j = i-1 \\
		\frac{1}{\theta_{n}^2} - \frac{1}{\theta_{n-1}^2} & i = j = 1 \\
		-\frac{1}{\theta_{n-i}^2} & i = j \geq 2\,,
	\end{cases}$ we have
	\[[(1/\theta_n^2)U(\theta_n, \dots, \theta_1)\lhat']_{i, j} = \begin{cases}
		\frac{1}{2\theta_{n-j+1}^2} - \frac{1}{2\theta_{n-j}^2} & i+2 \leq j \leq n-1 \\
		\frac{1}{2\theta_1^2} - \frac{1}{\theta_0^2} & i +2 \leq j = n \\
		\frac{\theta_{n-i+1}}{2\theta_{n-i}^2} - \frac{1}{2\theta_{n-i-1}^2} & i + 1 = j \leq n-1 \\
		\frac{\theta_2}{2\theta_1^2} - \frac{1}{\theta_0^2} & i + 1 = j = n \\
		-\frac{\theta_n}{2\theta_{n-1}^2} & i = j = 1 \\
		-\frac{2\theta_{n-i+1}-1}{2\theta_{n-i}^2} & 2 \leq i  = j \leq n-1 \\
		-\frac{\theta_1}{\theta_0^2} & i =j =n \\
		\frac{1}{2\theta_{n-i+1}} & i = j+1 \\
		0 & \text{else}\,.
	\end{cases}
	\]
	Similarly as before, we fix $j$ and solve an equation with respect to the $j$th column. We divide this step into two cases.
	
	\paragraph*{Typical columns ($2 \leq j \leq n-1$).} Here we consider a $j$th column of $(1/\theta_n^2)U(\theta_n, \dots, \theta_1)\lhat'$, which is given as
	\[\begin{bmatrix}
		\left(\frac{1}{2\theta_{n-j+1}^2} - \frac{1}{2\theta_{n-j}^2}\right)\bm{1}_{j-2} \\
		\frac{\theta_{n-j+2}}{2\theta_{n-j+1}^2} - \frac{1}{2\theta_{n-j}^2} \\
		-\frac{2\theta_{n-j+1}-1}{2\theta_{n-j}^2} \\
		\frac{1}{2\theta_{n-j}} \\
		\bm{0}_{n-j-1}
	\end{bmatrix}\,.
	\]
	This column, after premultiplying $U(\varphi_n, \dots, \varphi_1)^{-1}$, satisfies the following properties. The proof for the following lemma is similar to that of Lemma \ref{lem:ogmls} and is deferred until later.
	\begin{lemma}[Linear system for typical columns]\label{lem:ogmgls1}
		For $2 \leq j \leq n-1$, let $x'(j)$ be the unique solution $x' = [x'_1, \dots, x'_n]$ of
		\[U(\varphi_n, \dots, \varphi_1)x' = 2\theta_{n-j+1}\theta_{n-j}^2\begin{bmatrix}
			\left(\frac{1}{2\theta_{n-j+1}^2} - \frac{1}{2\theta_{n-j}^2}\right)\bm{1}_{j-2} \\
			\frac{\theta_{n-j+2}}{2\theta_{n-j+1}^2} - \frac{1}{2\theta_{n-j}^2} \\
			-\frac{2\theta_{n-j+1}-1}{2\theta_{n-j}^2} \\
			\frac{1}{2\theta_{n-j}} \\
			\bm{0}_{n-j-1}
		\end{bmatrix} = \begin{bmatrix}
			-\bm{1}_{j-2} \\
			\theta_{n-j+2}\theta_{n-j+1} - \theta_{n-j+2} - \theta_{n-j+1} \\
			-(2\theta_{n-j+1}-1)\theta_{n-j+1} \\
			\theta_{n-j+1}\theta_{n-j} \\
			\bm{0}_{n-j-1}
		\end{bmatrix}\,.
		\]
		Then the following hold.
		\begin{itemize}
			\item[(a)] $x'_n = \dots = x'_{j+2} = 0, x'_{j+1} > 0, x'_j < 0$ and $x'_{j-1} > 0$.
			\item[(b)] For $j \geq 3$, $x'_{j-2} < 0$ and $x'_k = \frac{\varphi_{n-k}-1}{\varphi_{n-k+1}}x'_{k+1}$ for all $1 \leq k \leq j-3$.
			\item[(c)] $-\theta_{n-j+1}^2 + \frac{1}{2}x'_{j-1} - \frac{\theta_{n-j+1}}{2\theta_{n-j}}x'_j \geq 0$.
			\item[(d)] For $j \geq 3$, $\theta_{n-j}^2 + \frac{\theta_{n-j+1}}{2\theta_{n-j+2}}x'_{j-2} - \frac{1}{2}x'_{j-1} \geq 0$.
		\end{itemize}
	\end{lemma}

	Note that $[U_n^{-1}\diag(1/(2\theta_{n-1}), \dots, 1/(2\theta_0))]_{i, j} = \begin{cases}
		\frac{1}{2\theta_{n-i}} & i = j \\
		-\frac{1}{2\theta_{n-j}} & j = i + 1 \\
		0 & \text{else}
	\end{cases}$ holds. Thus with Lemma \ref{lem:ogmgls1} and $x' = x'(j)$, the $(i, j)$ entry of $(1/\theta_n^2)\widetilde{H}^{-1}\lhat' = (1/\theta_n^2)U_n^{-1}\diag(1/(2\theta_{n-1}), \dots, 1/(2\theta_0))U_n(\varphi_n, \dots, \varphi_1)^{-1}U(\theta_n, \dots, \theta_1)\lhat'$ is given as
	\[\begin{cases}\frac{1}{2\theta_{n-j+1}\theta_{n-j}^2}\left(\frac{1}{2\theta_{n-i}}x'_i - \frac{1}{2\theta_{n-i-1}}x'_{i+1}\right) &  1 \leq i \leq n-1 \\
		\frac{1}{2\theta_{n-j+1}\theta_{n-j}^2}\left(\frac{1}{2\theta_0}x'_n\right) \geq 0 & i = n\,.
	\end{cases}
	\]
	In particular, except for $i \in \{j-2, j\}$, the $(i, j)$ entry of $(1/\theta_n^2)\widetilde{H}^{-1}\lhat'$ is nonnegative as
	\begin{align*}
		\frac{1}{2\theta_{n-i}}x'_i - \frac{1}{2\theta_{n-i-1}}x'_{i+1} \geq \left(\frac{1}{2\theta_{n-i}}-\frac{1}{2\theta_{n-i-1}}\right)x'_{i+1} &\geq 0, \ 1 \leq i \leq j-3,\\
		\frac{1}{2\theta_{n-j+1}}x'_{j-1} - \frac{1}{2\theta_{n-j}}x'_j &\geq 0, \\
		\frac{1}{2\theta_{n-j-1}}x'_{j+1} &\geq 0.
	\end{align*}
	by Lemma \ref{lem:ogmgls1} (a) and (b).

	Summing up what we have shown so far, for $(1/\theta_n^2)(\ltilde + \widetilde{H}^{-1}\lhat)$ it suffices to show that the $(i, j)$ entry is nonnegative for $i \in \{j-2, j-1\}$. For $i = j-2$, this entry is given as (from Lemma \ref{lem:ogmgls1} (c))
	\[\frac{1}{2\theta_{n-j+1}^2} + \frac{1}{2\theta_{n-j+1}\theta_{n-j}^2}\left(\frac{1}{2\theta_{n-j+2}}x'_{j-2} - \frac{1}{2\theta_{n-j+1}}x'_{j-1}\right) = \frac{1}{2\theta_{n-j+1}^2 \theta_{n-j}^2}\left(\theta_{n-j}^2 + \frac{\theta_{n-j+1}}{2\theta_{n-j+2}}x'_{j-2} - \frac{1}{2}x'_{j-1}\right) \geq 0\,,
	\]
	and for $i = j-1$ this entry is given as (from Lemma \ref{lem:ogmgls1} (d))
	\[-\frac{1}{2\theta_{n-j}^2} + \frac{1}{2\theta_{n-j+1}\theta_{n-j}^2}\left(\frac{1}{2\theta_{n-j+1}}x'_{j-1} - \frac{1}{2\theta_{n-j}}x'_{j}\right) = \frac{1}{2\theta_{n-j+1}^2\theta_{n-j}^2}\left(-\theta_{n-j+1}^2 + \frac{1}{2}x'_{j-1} - \frac{\theta_{n-j+1}}{2\theta_{n-j}}x'_j\right) \geq 0\,.
	\]
	
	\paragraph*{Special columns ($j \in \{1, n\}$).}
	
	The first column of $(1/\theta_n^2)U(\theta_n, \dots, \theta_1)\lhat'$ is given as
	\[\begin{bmatrix}
		-\frac{\theta_{n}}{2\theta_{n-1}^2} \\
		\frac{1}{2\theta_{n-1}} \\
		\bm{0}_{n-2}
	\end{bmatrix}\,,
	\]
	from which it is trivial to check that the $(i, 1)$ entry of $(1/\theta_n^2)\widetilde{H}^{-1}\lhat$ is nonnegative for all $2 \leq i \leq n$ (in particular, is equal to 0 for all $3 \leq i \leq n$).

	The $n$th column of $(1/\theta_n^2)U(\theta_n, \dots, \theta_1)\lhat'$ is given as
	\[\begin{bmatrix}
		(\frac{1}{2\theta_1^2} - \frac{1}{\theta_0^2})\bm{1}_{n-2} \\
		\frac{\theta_{2}}{2\theta_1^2}- \frac{1}{\theta_0^2} \\
		-\frac{\theta_1}{\theta_0}
	\end{bmatrix}\,.
	\]
	As in the case of typical columns, we solve the corresponding linear system as follows.
	\begin{lemma}[Linear system for $n$th column]\label{lem:ogmgls2}
		Let $y' = [y'_1, \dots, y'_n]$ be the unique solution of
		\[U(\varphi_n, \dots, \varphi_1)y' = 2\theta_1^2\begin{bmatrix}
			(\frac{1}{2\theta_1^2} - \frac{1}{\theta_0^2})\bm{1}_{n-2} \\
			\frac{\theta_{2}}{2\theta_1^2}- \frac{1}{\theta_0^2} \\
			-\frac{\theta_1}{\theta_0}
		\end{bmatrix} = \begin{bmatrix}
			-(2\theta_1^2 - 1)\bm{1}_{n-2} \\
			\theta_2 - 2\theta_1^2 \\
			-2\theta_1^3
		\end{bmatrix}\,.
		\]
		Then the following hold.
		\begin{itemize}
			\item[(a)] $y'_n = -4\theta_1 < 0, y'_{n-1} = \frac{1}{\varphi_2}(\theta_2 + 2\theta_1 - 2) > 0, y'_{n-2} = -\frac{1}{\varphi_3}(\theta_2 - 1 - (\varphi_2-1)y'_{n-1}) < 0$.
			\item[(b)] $y'_k = \frac{\varphi_{n-k}-1}{\varphi_{n-k+1}}y'_{k+1}$ for all $1 \leq k \leq n-3$.
		\end{itemize}
	\end{lemma}
	
	For $1 \leq i \leq n-3$, the $(i, n)$ entry of $(1/\theta_n^2)\wt{H}^{-1}\lhat'$ is nonnegative from $\frac{1}{2\theta_{n-i}}y'_i - \frac{1}{2\theta_{n-i-1}}y'_{i+1} \geq \left(\frac{1}{2\theta_{n-i}} - \frac{1}{2\theta_{n-i-1}}\right)y'_{i+1} \geq 0$. 
	
	The $(n-2, n)$ entry of $(1/\theta_n^2)(\ltilde + \widetilde{H}^{-1}\lhat)$ is $\frac{1}{2\theta_1^2}\left(\frac{1}{2\theta_2}y'_{n-2} - \frac{1}{2\theta_1}y'_{n-1}\right) + \frac{1}{2\theta_1^2} \geq \frac{1}{2\theta_1^3}\left(\theta_1 - \frac{1}{2}y'_{n-1}+ \frac{1}{2}y'_{n-2}\right)$ where
	\begin{align*}
		\theta_1 - \frac{1}{2}y'_{n-1} + \frac{1}{2}y'_{n-2} &= \theta_1 - \frac{1}{2}y'_{n-1} - \frac{1}{2\varphi_3}(\theta_2 - 1 - (\varphi_2-1)y'_{n-1}) \\
		&\geq \theta_1 - \frac{1}{2}y'_{n-1} - \frac{3}{8}(\theta_{2} - 1 - (\varphi_2-1)y'_{n-1}) \\
		&= \theta_1 - \frac{3}{8}(\theta_2-1)+\frac{3\varphi_2 - 7}{8}y'_{n-1} \\
		&= \theta_1 - \frac{3}{8}(\theta_2-1) + \left(\frac{3}{8}-\frac{7}{8\varphi_2}\right)(\theta_2 +2\theta_1 - 2) \\
		&\geq \theta_1 - \frac{3}{8}(\theta_2-1) - \frac{9}{32}(\theta_2 + 2\theta_1 - 2) \\
		&= \frac{7}{16}\theta_1 - \frac{21}{32}\theta_2 + \frac{15}{16} \\
		&\geq \frac{7}{16}\theta_1 - \frac{21}{32}\left(\theta_1 + \frac{3}{4}\right) + \frac{15}{16} \\
		&\geq \frac{7(2 - \theta_1)}{32} \geq 0\,,
	\end{align*}
	and the $(n-1, n)$ entry of $(1/\theta_n^2)(\ltilde + \widetilde{H}^{-1}\lhat)$ is $\frac{1}{2\theta_1^2}\left(\frac{1}{2\theta_1}y'_{n-1} - \frac{1}{2\theta_0}y'_n\right) - \frac{1}{2\theta_0^2} \geq \frac{1}{2\theta_1^3}\left(\frac{1}{2}y'_{n-1} - \frac{3}{4}y'_n - \theta_1^3\right)$ where
	\begin{align*}
		-\theta_1^3 + \frac{1}{2}y'_{n-1} - \frac{3}{4}y'_n &= -\theta_1^3 + \frac{1}{2\varphi_2}(\theta_2 +2\theta_1 - 2) + 3\theta_1 \\
		&= \theta_1 - 1 + \frac{1}{2\varphi_2}(\theta_2 + 2\theta_1 - 2) \geq 0 \,.
	\end{align*}
	
	\paragraph*{Nonnegativity of $\mu_{0, j}$.}
	Finally, for $\begin{bmatrix}\mu_{0, 1} & \dots & \mu_{0, n} \end{bmatrix} = -\bm{1}_n^T(\ltilde' + \widetilde{H}^{-1}\lhat')$, we have
	\[-(1/\theta_n^2)\bm{1}_n^T\ltilde' = \begin{bmatrix}
		\frac{1}{\theta_n^2} - \frac{1}{2\theta_{n-1}^2} & \frac{1}{2\theta_{n-1}^2} & \bm{0}_{n-2}^T
	\end{bmatrix}
	\]
	and 
	\begin{align*}
		-(1/\theta_n^2)\bm{1}_n^T(\widetilde{H}^{-1}\lhat') &= -\bm{1}_n^T U_n^{-1}\diag(1/(2\theta_{n-1}), \dots, 1/(2\theta_0)) (1/\theta_n^2)U(\varphi_n, \dots, \varphi_1)^{-1}U(\theta_n, \dots, \theta_1)\lhat' \\
		&= -\frac{1}{2\theta_{n-1}}e_1^T  (1/\theta_n^2)U(\varphi_n, \dots, \varphi_1)^{-1}U(\theta_n, \dots, \theta_1)\lhat'\,.
	\end{align*}
	We know that $(1, j)$ entry of $-(1/\theta_n^2)U(\varphi_n, \dots, \varphi_1)^{-1}U(\theta_n, \dots, \theta_1)\lhat'$ is nonnegative for all $j \neq 2$. Thus it suffices to verify the nonnegativity of the first two entries of $-(1/\theta_n^2)\bm{1}_n^T(\ltilde' + \widetilde{H}^{-1}\lhat')$. The first entry is $\frac{1}{\theta_n^2} - \frac{1}{2\theta_{n-1}^2} + \frac{4}{\theta_{n-1}^3\varphi_n}\left(\theta_n + \frac{\theta_{n-1}}{\varphi_{n-1}}\right) = \frac{1}{2\theta_{n-1}^3}\left(\frac{1}{2\varphi_n}\left(\theta_n + \frac{\theta_{n-1}}{\varphi_{n-1}}\right)-\theta_{n-1} + \frac{2\theta_{n-1}^3}{\theta_n^2}\right)$ where from Lemma \ref{lem:thetaphi},
	\begin{align*}
		\frac{1}{2\varphi_n}\left(\theta_n + \frac{\theta_{n-1}}{\varphi_{n-1}}\right)-\theta_{n-1} + \frac{2\theta_{n-1}^3}{\theta_n^2} &\geq \frac{1}{2+\sqrt{2}}\left(\theta_n + \frac{\theta_{n-1}}{\varphi_{n-1}}\right)-\theta_{n-1} + \frac{(\theta_{n}-1)\theta_{n-1}}{\theta_n} \\
		&\geq \frac{1}{2+\sqrt{2}}\left(\sqrt{2}\theta_{n-1} + \frac{2}{3}\theta_{n-1}\right)-\theta_{n-1} + \frac{1}{2}\theta_{n-1} \geq 0 \,.
	\end{align*}
	The second entry, for $n \geq 3$, is (letting $x' = x'(2)$) $-\frac{1}{4\theta_{n-1}^2\theta_{n-2}^2}x'_1 + \frac{1}{2\theta_{n-1}^2} = \frac{1}{2\theta_{n-1}^2\theta_{n-2}^2}(-\frac{1}{2}x'_1 + \theta_{n-2}^2)$ where
	\begin{align*}
		-\frac{1}{2}x'_{1} + \theta_{n-2}^2 &= -\frac{1}{2\varphi_n}(\theta_n(\theta_{n-1}-1)+2\theta_{n-2}^2 + (\varphi_{n-1}-1)x'_2) + \theta_{n-2}^2 \\
		&\geq - \frac{5}{16}\left(\theta_n(\theta_{n-1}-1)+2(\theta_{n-1}^2 - \theta_{n-1}) - \frac{\varphi_{n-1}-1}{\varphi_{n-1}}(2\theta_{n-1}^2 - \theta_{n-1}+x'_3)\right) + (\theta_{n-1}^2 - \theta_{n-1}) \\
		&\geq -\frac{5}{16}\left(\frac{10}{7}(\theta_{n-1}^2 - 1) + 2(\theta_{n-1}^2 - \theta_{n-1}) - \frac{1}{4}(2\theta_{n-1}^2 - \theta_{n-1}) - \frac{1}{4}x'_3\right)+(\theta_{n-1}^2 - \theta_{n-1}) \\
		&\geq -\frac{5}{16}\left(\frac{10}{7}(\theta_{n-1}^2 - 1) + 2(\theta_{n-1}^2 - \theta_{n-1}) - \frac{1}{4}(2\theta_{n-1}^2 - \theta_{n-1}) - \frac{1}{4}\theta_{n-1}\right)+(\theta_{n-1}^2 - \theta_{n-1}) \\
		&= \frac{1}{112}(9\theta_{n-1}^2 -42\theta_{n-1} + 50) \geq 0\,.
	\end{align*}
	Finally, the second entry for $n = 2$ is $-\frac{1}{4\theta_1^3}y'_1 + \frac{1}{2\theta_1^2} = \frac{1}{2\theta_1^3}\left(-\frac{1}{2}y_1' + \theta_1\right)$ where
	\begin{align*}
		-\frac{1}{2}y_1' + \theta_1 &= -\frac{1}{2\varphi_2}(\theta_2 +2\theta_1 - 2) + \theta_1 \\
		&\geq -\frac{1}{3}(\theta_2 +2\theta_1 - 2) + \theta_1 \\
		&= \frac{1}{3}(\theta_1-\theta_2 + 2) \\
		&\geq \frac{1}{3}\left(\theta_1 - \frac{3}{2}\theta_1 -\frac{1}{2} + 2\right) \geq 0 \,,
	\end{align*}
	where in the second inequality we used $\theta_2 = \frac{1+\sqrt{1+8\theta_1^2}}{2} \leq \frac{1+3\theta_1}{2}$.
	
	For item (ii), we can take $\xi' = 1 - \frac{\lambda'_{n-1, n} + \lambda'_{n, n-1}}{R'_n}$ where $\lambda'_{n-1, n} + \lambda'_{n, n-1} = \theta_n^2\left(\frac{1}{2\theta_0^2} + \frac{1}{2\theta_0^2} - \frac{1}{2\theta_1^2}\right) = \frac{\sqrt{5}+1}{4}\theta_n^2$ for $n \geq 2$. For $n = 1$, $\lambda'_{n-1, n} + \lambda'_{n, n-1} = \frac{1}{\theta_0^2} - \frac{1}{\theta_1^2} = \frac{3}{4}$.
	
\end{proof}

We present the proofs of the technical lemmas within the previous parts as follows.

\begin{proof}[Proof of Lemma \ref{lem:ogmgls1}]
	For notational convenience, let $i := n-j$ (which implies $1 \leq i \leq n-2$).
	\begin{itemize}
		\item[(a)] Since $U(\varphi_n, \dots, \varphi_1)$ is upper triangular, we can solve the equation iteratively starting from $x'_n$, which yields $x'_n = \dots = x'_{j+2} = 0$. Also,
		\begin{align*}
			x'_{j+1} &= \frac{\theta_{i+1}\theta_i}{\varphi_i} > 0, \\
			x'_j &= -\frac{1}{\varphi_{i+1}}(2\theta_{i+1}^2 - \theta_{i+1} + x'_{j+1}) < 0, \\
			x'_{j-1} &= \frac{1}{\varphi_{i+2}}(\theta_{i+2}\theta_{i+1} - \theta_{i+2}-\theta_{i+1} - x'_j - x'_{j+1}) \\
			&= \frac{1}{\varphi_{i+2}}(\theta_{i+2}\theta_{i+1} - \theta_{i+2}-\theta_{i+1} + (2\theta_{i+1}^2 - \theta_{i+1}) + (\varphi_{i+1}-1)x'_j) \\
			&= \frac{1}{\varphi_{i+2}}(\theta_{i+2}(\theta_{i+1}-1)+2\theta_i^2 + (\varphi_{i+1}-1)x'_j),
		\end{align*}
		where from $(\varphi_{i+1}-1)x'_j = -\frac{\varphi_{i+1}-1}{\varphi_{i+1}}(2\theta_{i+1}^2 - \theta_{i+1} + x_{j+1}) > -\frac{1}{2}(2\theta_{i+1}^2-\theta_{i+1} + x'_{j+1}) > -\frac{1}{2}(3\theta_{i+1}^2-\theta_{i+1})$, $x'_{j-1} > \frac{1}{\varphi_{i+2}}\left(\left(\theta_{i+1}+\frac{1}{2}\right)(\theta_{i+1}-1)+2(\theta_{i+1}^2-\theta_{i+1}) - \frac{1}{2}(3\theta_{i+1}^2 - \theta_{i+1})\right) = \frac{1}{2\varphi_{i+2}}(3\theta_{i+1}^2-4\theta_{i+1}-1) > 0$.
		
		\item[(b)] Solving the equation yields
		\[x'_{k} = \frac{1}{\varphi_{n+1-k}}\left(-1-\sum_{l=k+1}^{j+1} x'_l\right), \ k= j-2, \dots, 1\,,   
		\]
		which implies $x'_k = \frac{1}{\varphi_{n+1-k}}(-x'_{k+1} + \varphi_{n-k}x'_{k+1}) = \frac{\varphi_{n-k}-1}{\varphi_{n+1-k}}x'_{k+1}$ for all $1 \leq k \leq j-3$.
		From $x'_{j-2} = \frac{1}{\varphi_{i+3}}(-1-x'_{j-1}-x'_j - x'_{j+1})$, it suffices to show that $x'_{j-1} + x'_j + x'_{j+1} > -1$, which is equivalent to
		\[\frac{1}{\varphi_{i+2}}(\theta_{i+2}\theta_{i+1}-\theta_{i+2}-\theta_{i+1}) - \left(1-\frac{1}{\varphi_{i+2}}\right)\frac{1}{\varphi_{i+1}}(2\theta_{i+1}^2-\theta_{i+1}) + \left(1-\frac{1}{\varphi_{i+2}}\right)\left(1-\frac{1}{\varphi_{i+1}}\right)\frac{1}{\varphi_{i}}\theta_{i+1}\theta_i > -1\,.
		\]
		By Lemma \ref{lem:thetaphi}, $\frac{1}{\varphi_{i}} \geq \frac{2}{3}$ and $\frac{2}{3} \leq \frac{1}{\varphi_{i+1}}, \frac{1}{\varphi_{i+2}} \leq \frac{3}{4}$ (note that here, $i+2 < n$ as $j \geq 3$). Thus the left hand side is lower bounded by
		\begin{align*}
			&\frac{2}{3}(\theta_{i+2}\theta_{i+1}-\theta_{i+2}-\theta_{i+1}) - \left(1-\frac{2}{3}\right)\frac{3}{4}(2\theta_{i+1}^2 - \theta_{i+1}) + \left(1-\frac{3}{4}\right)^2\frac{2}{3}\theta_{i+1}\theta_{i} \\
			&=\frac{1}{24}(16\theta_{i+2}(\theta_{i+1}-1) - 12\theta_{i+1}^2 - 10\theta_{i+1} + \theta_{i+1}\theta_i) \\
			&\geq \frac{1}{24}((16\theta_{i+1}+8)(\theta_{i+1}-1)-12\theta_{i+1}^2 - 10\theta_{i+1} + \theta_i^2) \\
			&= \frac{1}{24}(4\theta_{i+1}^2 - 18\theta_{i+1} + (\theta_{i+1}^2 - \theta_{i+1})) \\ 
			&= \frac{1}{24}(5\theta_{i+1}^2 - 19\theta_{i+1}) > -1\,,
		\end{align*}
		as desired.
		
		\item[(c)] First, note that $x'_{j-1} > 0$ and $x'_j < 0$ from (a). Thus
		\begin{align*}
			-\theta_{i+1}^2 + \frac{1}{2}x'_{j-1} - \frac{\theta_{i+1}}{2\theta_{i}}x'_j &\geq -\theta_{i+1}^2 + \frac{1}{2}(x'_{j-1} - x'_j) \\
			&= -\theta_{i+1}^2 + \frac{1}{2}\frac{1}{\varphi_{i+2}}(\theta_{i+2}(\theta_{i+1}-1)+2\theta_i^2 + (\varphi_{i+1}-1)x'_j) - \frac{1}{2}x'_j \,.
		\end{align*}
		Consider first the case where $i \neq n-2$. Then $\frac{1}{\varphi_{i+2}} \geq \frac{2}{3}$, implying
		\begin{align*}
			&-\theta_{i+1}^2 + \frac{1}{2}\frac{1}{\varphi_{i+2}}(\theta_{i+2}(\theta_{i+1}-1)+2\theta_i^2 + (\varphi_{i+1}-1)x'_j) - \frac{1}{2}x'_j \\
			&\geq -\theta_{i+1}^2 + \frac{1}{3}\left(\left(\theta_{i+1}+\frac{1}{2}\right)(\theta_{i+1}-1)+2(\theta_{i+1}^2 - \theta_{i+1}) + \frac{1}{2}x'_j\right) - \frac{1}{2}x'_j \\
			&= -\frac{5}{6}\theta_{i+1} - \frac{1}{6}-\frac{1}{3}x'_j\,.
		\end{align*}
		Since $-x'_j \geq \frac{2}{3}(2\theta_{i+1}^2 - \theta_{i+1}+x'_{j+1}) \geq \frac{2}{3}\left(2\theta_{i+1}^2-\theta_{i+1} + \frac{2}{3}\theta_{i+1}\theta_i\right) \geq \frac{4}{3}\theta_{i+1}^2$ (from $\theta_i \geq \theta_1 \geq \frac{3}{2}$), the final term is lower bounded by $\frac{4}{9}\theta_{i+1}^2 - \frac{5}{6}\theta_{i+1} - \frac{1}{6} \geq 0$ (from $\theta_{i+1} \geq \theta_2 \geq 1+\frac{\sqrt{5}}{2}$).
		
		If $i = n-2$, then $\frac{1}{\varphi_{i+2}} \geq 2-\sqrt{2}$ and $\theta_{i+2} \geq \sqrt{2}\theta_{i+1}$, implying
		\begin{align*}
			&-\theta_{i+1}^2 + \frac{1}{2}\frac{1}{\varphi_{i+2}}(\theta_{i+2}(\theta_{i+1}-1)+2\theta_i^2 + (\varphi_{i+1}-1)x'_j) - \frac{1}{2}x'_j \\
			&\geq -\theta_{i+1}^2 + \frac{1}{2+\sqrt{2}}\left(\sqrt{2}\theta_{i+1}(\theta_{i+1}-1)+2(\theta_{i+1}^2 -\theta_{i+1}) + \frac{1}{2}x'_j\right) - \frac{1}{2}x'_j \\
			&= -\theta_{i+1} - \frac{1}{2\sqrt{2}}x'_j \\
			&\geq \frac{\sqrt{2}}{3}\theta_{i+1}^2 - \theta_{i+1} \geq 0\,,
		\end{align*}
		where the last inequality is from $\theta_{i+1} \geq \theta_2 = \frac{1+\sqrt{7+2\sqrt{5}}}{2}$.
		
		\item[(d)] First, note that $x'_{j-2} < 0$ and $x'_{j-1} > 0$ from (a) and (b). Thus
		\begin{align*}
			\theta_{i}^2 + \frac{\theta_{i+1}}{2\theta_{i+2}}x'_{j-2} - \frac{1}{2}x'_{j-1} &\geq \theta_i^2 + \frac{1}{2}(x'_{j-2} - x'_{j-1}) \\
			&=\theta_i^2 + \frac{1}{2\varphi_{i+3}}(-1-x'_{j-1}-x'_j - x'_{j+1}) - \frac{1}{2}x'_{j-1}\,.
		\end{align*}
		From Lemma \ref{lem:thetaphi}, this is lower bounded by
		\begin{align*}
			&\theta_i^2 + \frac{5}{14}(-1-x'_{j-1}-x'_j - x'_{j+1})-\frac{1}{2}x'_{j-1}\\
			&= \theta_i^2 - \frac{5}{14}-\frac{6}{7}x'_{j-1} - \frac{5}{14}(\theta_{i+2}\theta_{i+1}-\theta_{i+2}-\theta_{i+1} - \varphi_{i+2}x'_{j-1}) \\
			&= \theta_i^2 - \frac{5}{14}(\theta_{i+2}-1)(\theta_{i+1}-1) + \frac{5\varphi_{i+2}-12}{14}x'_{j-1} \\
			&\geq \theta_i^2 - \frac{5}{14}(\theta_{i+2}-1)(\theta_{i+1}-1) - \frac{19}{70}(\theta_{i+2}(\theta_{i+1}-1)+2\theta_i^2 + (\varphi_{i+1}-1)x'_j) \\
			&= \frac{16}{35}\theta_i^2 - \frac{5}{14}(\theta_{i+2}-1)(\theta_{i+1}-1) - \frac{19}{70}\theta_{i+2}(\theta_{i+1}-1) + \frac{19}{70}\frac{\varphi_{i+1}-1}{\varphi_{i+1}}(2\theta_{i+1}^2 - \theta_{i+1}+x'_{j+1}) \\
			&\geq \frac{16}{35}\theta_i^2 - \frac{5}{14}(\theta_{i+2}-1)(\theta_{i+1}-1) - \frac{19}{70}\theta_{i+2}(\theta_{i+1}-1) + \frac{19}{280}\left(2\theta_{i+1}^2 - \theta_{i+1}+\frac{2}{3}\theta_i^2\right) \\
			&\geq \frac{16}{35}\theta_i^2 - \frac{5}{14}(\theta_{i+1}-\frac{1}{4})(\theta_{i+1}-1) - \frac{19}{70}(\theta_{i+1} + \frac{3}{4})(\theta_{i+1}-1) + \frac{19}{280}\left(2\theta_{i+1}^2 - \theta_{i+1} + \frac{2}{3}(\theta_{i+1}^2 - \theta_{i+1})\right) \\
			&\geq \frac{1}{105}(\theta_{i+1}^2 - 6\theta_{i+1} + 12) \geq 0\,.
		\end{align*}
	\end{itemize}
\end{proof}

\begin{proof}[Proof of Lemma \ref{lem:ogmgls2}]
	\begin{itemize}
		\item[(a)] The formulae for $y'_n, y'_{n-1}, y'_{n-2}$ can be derived straightforwardly by solving the equation, with $\theta_1^2 = \theta_1 + 1$ and $\varphi_1 = \frac{3+\sqrt{5}}{4} = \frac{\theta_1^2}{2}$. The only nontrivial inequality is $y'_{n-2} < 0$, which holds from $\theta_2 - 1 - (\varphi_2 - 1)y'_{n-1} = \theta_2-1-\left(1-\frac{1}{\varphi_2}\right)(\theta_2 + 2\theta_1-2) \geq \theta_2 - 1 - \frac{1}{3}(\theta_2+2\theta_1-2) = \frac{2}{3}\left(\theta_2 - \theta_1 - \frac{1}{2}\right) > 0$.
		\item[(b)] The proof follows as in Lemma \ref{lem:ogmgls1} (b).
	\end{itemize}
\end{proof}

\end{document}